\documentclass[reqno]{amsart}
\usepackage{amsmath,amsfonts,amssymb,amsthm,dsfont}
\usepackage[usenames,dvipsnames]{xcolor}
\usepackage[colorlinks=true]{hyperref} 
\usepackage{verbatim}
\usepackage{accents}
\usepackage{kantlipsum} 
\hypersetup{linkcolor=Violet,citecolor=PineGreen}
\newtheorem{thm}{Theorem}[section]
\newtheorem{lem}[thm]{Lemma}
\newtheorem{prop}[thm]{Proposition}
\newtheorem{cor}[thm]{Corollary}
\theoremstyle{definition}
\newtheorem{defn}[thm]{Definition}

\newtheorem{rmk}[thm]{Remark}
\numberwithin{equation}{section}
\newcommand{\N}{\mathbb{N}}
\newcommand{\R}{\mathbb{R}}
\newcommand{\Q}{\mathbb{Q}}

\newcommand{\K}{\mathcal{K}}
\newcommand{\M}{\mathcal{M}}
\newcommand{\T}{\mathcal{T}}
\newcommand{\U}{\mathcal{U}}
\newcommand{\ep}{\varepsilon}

\newcommand{\Hu}{\mathbb{H}}
\newcommand{\C}{\mathfrak{C}}
\newcommand{\LC}{\mathfrak{LC}}
\newcommand{\SC}{\mathfrak{SC}}

\newcommand{\nbd}{\nobreakdash}
\newcommand{\TC}{\Theta\C}

\newcommand{\WTC}{\mathfrak{W}\Theta\mathfrak{C}}
\newcommand{\nin}{{n\in\N}}
\newcommand{\nti}{{n\to\infty}}

\let\inf\relax \DeclareMathOperator*\inf{\vphantom{p}inf}
\DeclareMathOperator{\cls}{cls}
\newcommand{\meas}{\mathrm{meas}}
\DeclareMathOperator{\supp}{supp}
\DeclareMathOperator{\dist}{dist}
\DeclareRobustCommand{\subtitle}[1]{\\[2ex]\normalfont \emph{#1}}
\begin{document}
\title[Weak topologies for Carath\'eodory differential equations]
{Weak topologies for Carath\'eodory differential equations. Continuous dependence, exponential Dichotomy and attractors\subtitle{D\MakeLowercase{edicated to the memory of} G\MakeLowercase{eorge} S\MakeLowercase{ell}}}
\author[I.P.~Longo]{Iacopo P. Longo}
\address{Departamento de Matem\'{a}tica Aplicada, Universidad de
Valladolid, Paseo del Cauce 59, 47011 Valladolid, Spain.}
\email[Iacopo P. Longo]{iaclon@wmatem.eis.uva.es}
\email[Sylvia Novo]{sylnov@wmatem.eis.uva.es}
\email[Rafael Obaya]{rafoba@wmatem.eis.uva.es}
\thanks{Partly supported by MINECO/FEDER
under project MTM2015-66330-P and EU Marie-Sk\l odowska-Curie ITN Critical Transitions in
Complex Systems (H2020-MSCA-ITN-2014 643073 CRITICS)}
\author[S.~Novo]{Sylvia Novo}
\author[R.~Obaya]{Rafael Obaya}
\subjclass[2010]{}
\date{}
\begin{abstract}
We introduce new weak topologies and spaces of Carath\'{e}odory functions where the solutions of the ordinary differential equations depend continuously on the initial data and vector fields. The induced local skew-product flow is proved to be continuous, and a notion of linearized skew-product flow is provided.  Two applications are shown. First, the propagation of the exponential dichotomy over the trajectories of the linearized skew-product flow and the structure of the dichotomy or Sacker-Sell spectrum. Second, how particular bounded absorbing sets for the process defined by a Carath\'{e}odory vector field $f$ provide bounded pullback attractors for the processes with vector fields in the alpha-limit set, the omega-limit set or the whole hull of $f$. Conditions for the existence of a pullback or a global attractor for the skew-product semiflow, as well as application examples  are also given.
\end{abstract}
\maketitle
\section{Introduction}\label{secintro}
This paper contains the development of topological methods to study the local and global behaviour of the solutions of families of nonautonomous Carath\'{e}odory ordinary differential equations. In the first part of this work, we introduce new  weak topologies and  Carath\'{e}odory spaces  where the solutions of the differential equations depend continuously on the initial data.  Such contributions complete and extend some results obtained in Longo \emph{et al.}~\cite{paper:LNO}, where the  strong version of these topologies and spaces has been investigated. Thanks to the obtained continuity theorems,  the propagation of the exponential dichotomy  over the trajectories of the linearized semiflows as well as the structure of the corresponding  Sacker-Sell spectrum can be analyzed. In the last part of the paper, we provide conditions under which the existence of particular bounded absorbing sets for the processes defined by a suitable dissipative Carath\'{e}odory vector field $f$, allows to deduce the existence of bounded pullback attractors for the processes with vector field  belonging to either the alpha-limit set, the  omega-limit set, or the whole hull of $f$. Under appropriate assumptions, these theorems also provide the existence of a pullback or a global attractor for the induced skew-product semiflow.
\par\smallskip
 The study of the topologies of continuity for Carath\'{e}odory differential equations is a classical question considered by Artstein~\cite{paper:ZA1,paper:ZA2,paper:ZA3}, Heunis~\cite{paper:AJH}, Miller and Sell~\cite{book:RMGS, paper:RMGS1}, Neustadt~\cite{paper:LWN}, Opial~\cite{paper:O}, among many others. We continue the work started in~\cite{paper:LNO} introducing new dynamical arguments and  methods that allow a more exhaustive analysis of the qualitative behavior of the solutions and some of their important dynamical implications. A range of dynamical scenarios is opened where  it is possible to combine techniques of continuous skew-product flows, processes and random dynamical systems in order to obtain more precise dynamical information (see Arnold~\cite{book:LA}, Aulbach and Wanner~\cite{paper:AW}, Berger and Siegmund~\cite{paper:BS},  Caraballo and Han~\cite{book:CH}, Carvalho \emph{et al.}~\cite{book:CLR}, Johnson \emph{et al.}~\cite{book:JONNF}, P\"otzsche and Rasmussen~\cite{paper:CPMR}, Sacker and Sell~\cite{sase3, sase4, sase5, sase}, Sell~\cite{book:GS},  Shen and Yi~\cite{paper:WSYY} and the references therein).
 \par\smallskip
The structure and main results of the paper are organized as follows. Subsection~\ref{spatop} is devoted to the introduction of the spaces and topologies used along the paper. In particular, the classical weak topology  $\sigma_D$ and the classical strong topology  $\T_D$ on the space $\SC$ of strong Carath\'eodory functions, are recalled. Such topologies, as well as other weak and strong topologies, have been studied in~\cite{paper:ZA1,paper:ZA2,paper:ZA3,paper:AJH,paper:LWN,book:RMGS, paper:RMGS1}. Additionally, we recall the   strong topology $T_\Theta$ on the space $\TC$, presented in~\cite{paper:LNO}, and introduce a new weak topology $\sigma_\Theta$, defined on the new space $\WTC$ in terms of a countable family of moduli of continuity $\Theta$.  The  locally convex space $(\WTC,\sigma_\Theta)$ represents the weak version of the locally convex space $(\TC , \T_\Theta)$. Subsection~\ref{firtop} contains some preliminary topological results for such a new metric space. \par\smallskip
In Section~\ref{sec:cont-Flow} we prove that if $E \subset \WTC$  admits $L^1_{loc}$-equicontinuous $m$-bounds then the translation map is continuous on $(E,\sigma_\Theta)$. As a corollary, when $f\in\WTC$ has $L^1_{loc}$-equicontinuous $m$-bounds and $E$ is the Hull of $f$ in $(\WTC,\sigma_\Theta)$,
  one obtains a continuous flow on $E$. Furthermore, if $E \subset \LC$ admits $L^1_{loc}$-equicontinuous $m$-bounds, one has that the $m$-bounds determine a suitable set of moduli of continuity $\Theta$. Then, starting from $\Theta$ and considering any set $B \subset \WTC(\R^N)$ and $C \subset \WTC(\R^N)$ with $L^1_{loc}$-equicontinuous $m$-bounds, we prove the continuity of the solutions of the triangular Carath\'eodory systems
\begin{equation*}
\begin{cases}
\dot x=f(t,x), &x(0)=x_0\,,\\
\dot y=F(t,x)\,y+h(t,x), &y(0)= y_0\,,
\end{cases}
\end{equation*}
with respect to the initial data $(f,F,h,x_0,y_0)$ in the product space  $ (E, \sigma_\Theta)\times (B,\sigma_\Theta) \times (C, \sigma_\Theta) \times \R^N \times \R^N$. As a consequence, we deduce the continuity of the local skew-product flow given by the base flow $(t,g,G,k)\mapsto (g_t,G_t,k_t)$ on the Hull of $(f,F,h)$ in $(\LC\times\WTC(\R^{N\times N})\times\WTC,\sigma_\Theta\times\sigma_\Theta\times\sigma_\Theta)$, and by the solutions, $x(t,g,x_0)$ and $y(t,g,G,k,x_0,y_0)$, of the corresponding differential equations.\par\smallskip
Section~\ref{expdich} studies the linearized skew-product flow, defined when $f \in \LC$ is continuously differentiable with respect to $x$ and both $f$ and its jacobian $J_xf \in \SC(\R^{N\times N})$ has $L^1_{loc}$-equicontinuous $m$-bounds. Denoting by $\Hu$ the hull of $(f,J_xf)$ in the space $ (\LC\times\TC,\sigma_\Theta\times\sigma_\Theta)$ and by $\Omega=\Hu\times \R^N$, one can write
 \[\begin{array}{cccc}\Psi\colon& \R\times\Omega\times\R^N &\to& \Omega\times \R^N\\ & (t,g,G,x_0,y_0)&\mapsto& (g_t,G_t,x(t,g,x_0),y(t,g,G,x_0,y_0))\end{array}\]
 when the solutions of $\dot x=g(t,x)$ are globally defined.
For each  $\omega =( g,G,x_0)$ we show that, if  the linear system $\dot y=G(\omega_t)\,y=G(t,x(t,g,x_0))\,y$ has exponential dichotomy, then it has exponential dichotomy over  $\Hu(\omega)=\cls\{\omega_t\mid t\in \R\}$ and hence, the dynamical spectrum $\Sigma(\omega)=\Sigma(\Hu(\omega))$. From here we describe  $\Sigma(\omega)$ following the arguments  given in~\cite{sase} when $\Hu(\omega)$ is compact and in Siegmund~\cite{paper:SS} for general $L^1_{loc}$-coefficients, case in which $\Sigma(\omega)$ could be unbounded. However, when $x(t, g,x_0)$ is  bounded we deduce that the solutions of the linear system have bounded growth and $\Sigma(\omega)$ is the union of $k \leq N$ compact~intervals. An analogous analysis is also carried out for the strong topology $\T_\Theta$.
 \par\smallskip
 Section~\ref{sec-pullback} deals with pullback and global attractors for Carath\'{e}odory ODEs as an application of the continuity of the local skew-product flow. In particular, starting from specific properties on the solutions of an initial problem $\dot x= f(t,x)$, it is possible to obtain the existence of a bounded pullback attractor for the processes induced by systems with vector field in either the alpha limit set of $f$, the omega limit set of $f$, or the whole hull of $f$. Furthermore, conditions for the existence of pullback and global attractors for the induced skew-product flow are also provided.\par\smallskip
 Finally, Section~\ref{examples} provides sufficient conditions under which the results of the previous section can be applied. In fact, several types of attractors, both for the induced process and the induced skew-product flow, are obtained. In Subsection~\ref{scalar} the size of the solutions of a Carath\'eodory differential system $\dot x=f(t,x)$ is compared with the size of the solutions of a scalar linear equation, while in Subsection~\ref{system} a comparison with a system of linear Carath\'eodory equations is given.
\section{Topological preliminaries}\label{sectopo}
This section provides the topological set up for the entire paper. Initially, the topological spaces and the most important topological properties are introduced or recalled. Then, some preliminary results on such spaces are shown.
\subsection{Spaces and topologies}\label{spatop}
In the following, we will denote by $\R^N$ the $N$-dimensional euclidean space with norm $|\cdot|$ and by $B_r$ the closed ball of $\R^N$ centered at the origin and with radius $r$. When $N=1$ we will simply write $\R$ and the symbol $\R^+$ will denote the set of positive real numbers. Moreover, for any interval $I\subseteq\R$ and any $W\subset\R^N$, we will use the following notation
\begin{itemize}
\item[] $C(I,W)$: space of continuous functions from $I$ to $W$ endowed with the norm $\|\cdot\|_\infty$.
\item[] $C_C(\R)$: space of real continuous functions with compact support in $\R$, endowed with the norm $\|\cdot\|_\infty$. When we want to restrict to the positive continuous functions with compact support in $\R$, we will write  $C^+_C(\R)$.
\item[] $L^1(I,\R^N)$: space of measurable functions from $I$ to $\R^N$ whose norm is in the Lebesgue space $L^1(I)$.
\item[] $L^1_{loc}(\R^N)$: the space of all functions $x(\cdot)$ of $\R$ into $\R^N$ such that for every compact interval $I\subset\R$, $x(\cdot)$ belongs to $L^1\big(I,\R^N\big)$. When $N=1$, we will simply write $L^1_{loc}$.
\end{itemize}
We will consider, and denote by $\C(\R^M)$ (or simply $\C$ when $M=N$), the set of functions $f\colon\R\times\R^N\to \R^M$ satisfying
\begin{itemize}
\item[(C1)] $f$ is Borel measurable and
\item[(C2)] for every compact set $K\subset\R^N$ there exists a real-valued function $m^K\in L^1_{loc}$, called \emph{$m$-bound} in the following, such that for almost every $t\in\R$, one has $|f(t,x)|\le m^K(t)$ for any $x\in K$.
\end{itemize}
Now we introduce the sets of Carath\'eodory functions which are used in this work.
\begin{defn}\label{def:LC}
A function $f\colon\R\times\R^N\to \R^M$ is said to be \emph{Lipschitz Carath\'eodory}, and we will write $f\in \LC(\R^M)$ (or simply $f\in\LC$ when $M=N$), if it satisfies (C1), (C2) and
\begin{itemize}
\item[(L)] for every compact set $K\subset\R^N$ there exists a real-valued function $l^K\in L^1_{loc}$ such that $|f(t,x)-f(t,y)|\le l^K(t)|x-y|$ for any $x,y\in K$ and almost every $t\in\R$.
\end{itemize}
In particular, for any compact set $K\subset\R^N$, we refer to \emph{the optimal $m$-bound} and \emph{the optimal $l$-bound} of $f$ as to
\begin{equation}
m^K(t)=\sup_{x\in K}|f(t,x)|\qquad \mathrm{and}\qquad l^K(t)=\sup_{\substack{x,y\in K\\ x\neq y}}\frac{|f(t,x)-f(t,y)|}{|x-y|}\, ,
\label{eqOptimalMLbound}
\end{equation}
respectively. Clearly, for any compact set $K\subset\R^N$ the suprema in \eqref{eqOptimalMLbound}  can be taken for a countable dense subset of $K$ leading to the same actual definition, which grants that the functions defined in \eqref{eqOptimalMLbound} are measurable.
\end{defn}
\begin{defn}\label{def:SC}
A function $f\colon\R\times\R^N\to \R^M$ is said to be \emph{strong Carath\'eodory}, and we will write $f\in \SC(\R^M)$ (or simply $f\in\SC$ when $M=N$), if it satisfies (C1), (C2) and
\begin{itemize}
\item[(S)] for almost every $t\in\R$, the function $f(t,\cdot)$ is continuous.
\end{itemize}
The concept of \emph{optimal $m$-bound} for a strong  Carath\'eodory function on any compact set $K\subset\R^N$, is defined exactly as in equation \eqref{eqOptimalMLbound}.
\end{defn}
Functions which are not necessarily continuous in the second variable are also considered. First, we set some notation.
\begin{defn}
We call  \emph{a suitable set of moduli of continuity}, any countable  set of non-decreasing continuous functions
\begin{equation*}
\Theta=\left\{\theta^I_j \in C(\R^+, \R^+)\mid j\in\N, \ I=[q_1,q_2], \ q_1,q_2\in\Q\right\}
\end{equation*}
such that $\theta^I_j(0)=0$ for every $\theta^I_j\in\Theta$, and  with the relation of partial order given~by
\begin{equation*}\label{def:modCont}
\theta^{I_1}_{j_1}\le\theta^{I_2}_{j_2}\quad \text{whenever } I_1\subseteq I_2 \text{ and } j_1\le j_2 \, .
\end{equation*}
\end{defn}
Now we introduce the family of sets $\TC(\R^M)$ and $\WTC(\R^M)$, where $\Theta$ is a suitable set of moduli of continuity.
\begin{defn}\label{def:TCnWTC}
Let   $\Theta$ be a suitable set of moduli of continuity, and $\K_j^I$ the set of functions in $C(I,B_j)$ which admit $\theta^I_j$ as modulus of continuity.
\begin{itemize}
\item
We say that $f$ is \emph{$\Theta$\nbd-Carath\'eodory} and write $f\in \TC(\R^M)$ (or simply $f\in\TC$ when $M=N$), if $f$ satisfies  (C1), (C2), and for each $j\in\N$ and $I=[q_1,q_2], \ q_1,q_2\in\Q$ one has \vspace{0.2cm}
\begin{itemize}
\item[(T)] if $\big(x_n(\cdot)\big)_{\nin}$ is a sequence in $\K_j^I$  uniformly converging to $x(\cdot)\in\K_j^I$, then
\begin{equation*}
\qquad\qquad\lim_{\nti}\int_I\big|f\big(t,x_n(t)\big)-f\big(t,x(t)\big)\big|dt=0.
\label{T}
\end{equation*}
\end{itemize}
\item
We say that $f$ is \emph{weak $\Theta$\nbd-Carath\'eodory}, and write $f\in \WTC(\R^M)$ (or simply $f\in\WTC$ when $M=N$), if $f$ satisfies (C1), (C2) and\vspace{0.2cm}
\begin{itemize}
\item[(W)] for each $j\in\N$ and $I=[q_1,q_2], \ q_1,q_2\in\Q$, if $\big(x_n(\cdot)\big)_{\nin}$ is a sequence in $\K_j^I$  uniformly converging to $x(\cdot)\in\K_j^I$, then \vspace{-.025cm}
\begin{equation}
\qquad\qquad \lim_{\nti}\int_If\big(t,x_n(t)\big)\, dt=\int_If\big(t,x(t)\big)\, dt.
\label{WT}
\end{equation}
\end{itemize}
\end{itemize}
\end{defn}
As regards Definitions \ref{def:LC}, \ref{def:SC} and \ref{def:TCnWTC}, we identify the functions which lay in the same set and only differ on a negligible subset of  $\R^{1+N}$, following the same reasoning presented in \cite{paper:LNO}. The constraint about belonging to the same set is crucial. Indeed, without any additional constraint, a function in $\SC(\R^M)$ could actually be identified with a function which is not in $\SC(\R^M)$. Furthermore, such identifications imply that $\LC(\R^M)\subset\SC(\R^M)$ and $\TC(\R^M)\subset\WTC(\R^M)$, but $\SC(\R^M)$  is not included in $\TC(\R^M)$. Nevertheless, a continuous injection, which is not a bijection, of $\SC(\R^M)$ in $\TC(\R^M)$ is straightforward.  Thus, the following chain can be sketched\vspace{-.025cm}
\begin{equation}
\LC(\R^M)\subset\SC(\R^M)\hookrightarrow\TC(\R^M)\subseteq\WTC(\R^M)\, ,
\label{eq:SPincl}
\vspace{-.025cm}
\end{equation}
where $\Theta$ is any suitable set of moduli of continuity.\par \vspace{.025cm}
The following result characterizes the process of identification in $\WTC(\R^M)$ and, as a consequence, implies that $\WTC(\R^M)$ is a metric space when endowed with the topology defined immediately after. We skip the proof because it presents minor changes with respect to the one of Proposition 2.6 in  \cite{paper:LNO}.
\begin{prop}
Let $f,g\in\WTC(\R^M)$ coincide almost everywhere in $\R\times\R^N$. Then, for any $\K_j^I$ as in \rm{Definition~\ref{def:TCnWTC}}, we have that
\begin{equation*}\label{prop:coincideAE=>CoincideOnTheContinuous}
\forall\, x(\cdot)\in\K_j^I\ :\ f\big(t,x(t)\big)=g\big(t,x(t)\big) \text{ for a.e. } t\in I\, .
\end{equation*}
\end{prop}
We endow the previously introduced sets with suitable topologies. As a rule, when inducing a topology on a subspace we will denote the induced topology with the same symbol which denotes the topology on the original space. The space $\WTC(\R^M)$ will be endowed with the following topology.
\begin{defn}
Let $\Theta$ be a suitable set of moduli of continuity. We call $\sigma_{\Theta}$ the topology on $\WTC(\R^M)$ generated by the family of seminorms \vspace{-.025cm}
\begin{equation*}
p_{I,\, j}(f)=\sup_{x(\cdot)\in\K_j^I}\left|\,\int_If\big(t,x(t)\big)\,dt\,\right| ,\quad f\in\WTC(\R^M)\, ,
\vspace{-.025cm}
\end{equation*}
with $I=[q_1,q_2]$, $q_1,q_2\in\Q$, $j\in\N$, and $\K_j^I$  as in Definition~\ref{def:TCnWTC}.
$\left(\WTC(\R^M),\sigma_{\Theta}\right)$ is a locally convex metric space.
\end{defn}
On $\SC(\R^M)$, we introduce the following topology
\begin{defn}\label{def:SigmaD}
Let $D$ be a countable and dense subset of $\R^N$. We call $\sigma_{D}$ the topology on $\SC(\R^M)$ generated by the family of seminorms
\begin{equation*}
p_{I,\, x}(f)=\left|\,\int_If(t,x)\, dt\,\right|,\quad f\in\SC(\R^M), \, x\in D,\, I=[q_1,q_2],\,  q_1,q_2\in\Q\, .
\end{equation*}
$\left(\SC(\R^M),\sigma_{D}\right)$ is a  locally convex metric space.
\end{defn}
Notice that $\SC(\R^M)$ and $\LC(\R^M)$ can be endowed with both previous topologies and the following chain of order holds
\begin{equation}
\sigma_D\le\sigma_{\Theta}\, .
\label{eq:TopIncl}
\end{equation}
We also recall two strong topologies for the spaces $\TC$ and $\SC$ respectively.
\begin{defn}
Let $\Theta$ be a suitable set of moduli of continuity. We call $\T_{\Theta}$ the topology on $\TC(\R^M)$ generated by the family of seminorms
\begin{equation*}
p_{I,\, j}(f)=\sup_{x(\cdot)\in\K_j^I}\int_I\big|f\big(t,x(t)\big)\big|\,dt ,\quad f\in\TC(\R^M)\, ,
\end{equation*}
with $I=[q_1,q_2]$, $q_1,q_2\in\Q$, $j\in\N$, and $\K_j^I$  as in Definition~\ref{def:TCnWTC}.
$\left(\TC(\R^M),\T_{\Theta}\right)$ is a locally convex metric space.
\end{defn}
\begin{defn}\label{def:TD}
Let $D$ be a countable and dense subset of $\R^N$. We call $\T_{D}$ the topology on $\SC(\R^M)$ generated by the family of seminorms
\begin{equation*}
p_{I,\, x}(f)=\int_I|\,f(t,x)\,|\,dt,\quad f\in\SC(\R^M), \, x\in D,\, I=[q_1,q_2],\,  q_1,q_2\in\Q\, .
\end{equation*}
$\left(\SC(\R^M),\T_{D}\right)$ is a  locally convex metric space.
\end{defn}
Notice, once again, that the space $\LC$ can be endowed with both $\T_\Theta$ and $\T_D$ and also that $\sigma_{\Theta}\le\T_\Theta$.
\par
Finally, we recall the notions of $L^1_{loc}$-equicontinuity and $L^1_{loc}$-boundedness and prove some results on the previously outlined topological spaces once such properties are assumed to hold.
A subset $S$ of positive functions in $L^1_{loc}$  is bounded if for every $r>0$ the following inequality holds
\begin{equation*}
\sup_{m\in S}\int_{-r}^r m(t)\,  dt<\infty\,.
\end{equation*}
In such a case we will say that $S$ is $L^1_{loc}$-bounded.
\begin{defn}\label{weakcomp}
A set $S$ of positive functions in $L^1_{loc}$ \emph{is $L^1_{loc}$-equicontinuous} if for any $r>0$ and for any $\ep>0$ there exists a $\delta=\delta(r,\ep)>0$ such that, for any $-r\le s\le t\le r$, with $t-s<\delta$, the following inequality holds
\begin{equation*}
\sup_{m\in S}\int_{s}^t m(u)\,du<\ep\, .
\end{equation*}
\end{defn}
\begin{rmk}\label{rmk:equicnt=>bound}
Notice that the $L^1_{loc}$-equicontinuity implies the $L^1_{loc}$-boundedness.
\end{rmk}
The following definition extends the previous notions to sets of Carath\'eodory functions through their $m$-bounds and/or $l$-bounds. By time translation at time $t$ of a function $f$ we mean the application $f_t\colon\R\times\R^N\to\R^M$ defined by $(s,x)\mapsto f_t(s,x)=f(s+t,x)$.
\begin{defn}
We say that
\begin{itemize}
\item[(i)] a set $E\subset\C(\R^M)$ \emph{has $L^1_{loc}$-bounded (resp. $L^1_{loc}$-equicontinuous)} $m$-bounds, if for any $j\in\N$ there exists a set  $S^j\subset L^1_{loc}$ of $m$-bounds of the functions of $E$ on $B_j$, such that $S^j$ is $L^1_{loc}$\nbd-bounded (resp. $L^1_{loc}$-equicontinuous);
\item[(ii)] \emph{$f\in\C(\R^M)$ has $L^1_{loc}$-bounded (resp. $L^1_{loc}$-equicontinuous) $m$-bounds} if the set $\{f_t\mid t\in\R\}$ admits  $L^1_{loc}$-bounded (resp. $L^1_{loc}$-equicontinuous) $m$-bounds;
\item[(iii)]a set $E\subset\LC(\R^M)$  \emph{has  $L^1_{loc}$-bounded  (resp. $L^1_{loc}$-equicontinuous)} $l$\nbd-bounds, if for any $j\in\N$, the set $S^j\subset L^1_{loc}$, made up of the optimal $l$-bounds on $B_j$  of the functions in $E$,  is  $L^1_{loc}$-bounded  (resp. $L^1_{loc}$-equicontinuous);
\item[(iv)] $f\in\LC(\R^M)$  \emph{has $L^1_{loc}$-bounded  (resp. $L^1_{loc}$-equicontinuous) $l$-bounds} if  the set $\{f_t\mid t\in\R\}$  has $L^1_{loc}$-bounded (resp. $L^1_{loc}$-equicontinuous) $l$-bounds.
\end{itemize}
\label{def:05.07-13:05}
\end{defn}
\subsection{First topological results}\label{firtop}
Some preliminary information on the previously outlined topological spaces is given. In particular, we also generalize some of the results given in \cite[Section 4]{paper:LNO}. Firstly, notice that, if a function $x(\cdot)$ belongs to the set $\K_j^I$ given in definition~\ref{def:TCnWTC}, and we take $p_1$, $p_2\in\Q$ such that $J=[p_1,p_2]\subset I$, then $x(\cdot)$ does not necessarily belong to $\K_j^J$. Thus, the next technical lemma is needed.
\begin{lem}\label{lem:conv-subint}
Let $\Theta$ be a suitable set of moduli of continuity.
\begin{itemize}
\item[(i)] Let $f$ be a function of $\,\WTC(\R^M)$. For each $j\in\N$ and  $I=[q_1,q_2]$, $q_1,q_2\in\Q$, if $\big(x_n(\cdot)\big)_{\nin}$ is a sequence in $\K_j^I$  uniformly converging to $x(\cdot)\in\K_j^I$, then
\begin{equation*}
\lim_{\nti}\int_{p_1}^{p_2} f\big(t,x_n(t)\big)\, dt=\int_{p_1}^{p_2} f\big(t,x(t)\big)\, dt.
\end{equation*}
whenever $p_1$, $p_2\in \Q$ and $q_1\le p_1 < p_2\le q_2$. \smallskip
\item[(ii)]  Let $(g_n)_\nin$ be a sequence in $\WTC(\R^M)$ converging to a function $g$ in $\left(\WTC(\R^M),\sigma_\Theta\right)$. Then, for each $I=[q_1,q_2]$, $q_1,q_2\in\Q$ and $j\in\N$
\[ \lim_{n\to\infty}\sup_{x(\cdot)\in\K_j^I}\left| \int_{p_1}^{p_2} \big[g_n\big(t,x(t)\big)- g\big(t,x(t)\big)\big]\,dt\right|=0\]
whenever $p_1$, $p_2\in \Q$ and $q_1\le p_1 < p_2\le q_2$.
\end{itemize}
\end{lem}
\begin{proof} First, for every $\nin$, we define $\widehat x_n\colon I\to B_j$ by
$\widehat x_n(t)=x_n(p_1)$ if  $t\in [q_1,p_1]$, $\widehat x_n(t)=x_n(t)$ if $t\in [p_1,p_2]$, and $\widehat x_n(t)=x_n(p_2)$ if $t\in [p_2,q_2]$.
 Notice that for all $\nin$ one can write
 \begin{equation}\label{eq:decom}
 \int_{p_1}^{p_2}\!\!\!\! f\big(t,x_n(t)\big)\,dt= \int_{q_1}^{q_2} \!\!\!\!f\big(t,\widehat x_n(t)\big)\,dt
 -\int_{q_1}^{p_1} \!\!\!\!f\big(t,x_n(p_1)\big)\,dt -\int_{p_2}^{q_2}\!\!\!\! f\big(s,x_n(p_2)\big)\,dt\,.
 \end{equation}
 The sequence $\big(\widehat x_n(\cdot)\big)_{\nin}$ is in $\K_j^I$ and converges uniformly to the function $\widehat x\colon I\to \R^N$ defined by $\widehat x(t)=x(p_1)$ if   $t\in [q_1,p_1]$, $\widehat x(t)= x(t)$ if $t\in [p_1,p_2]$ and $\widehat x(t)= x(p_2)$ if   $t\in [p_2,q_2]$. Moreover, the sequences of constant functions  $\big(x_n(p_1)\big)_{\nin}$ and $\big(x_n(p_2)\big)_{\nin}$   respectively belong to $\K_j^{[q_1,p_1]}$ and $\K_j^{[p_2,q_2]}$, and converge to the constant functions given by $x(p_1)$ and $x(p_2)$ respectively. Therefore, statement (i) follows from~\eqref{eq:decom} and \eqref{WT}. The proof of (ii) is obtained through similar reasonings and using the definition of convergence in $\left(\WTC(\R^M),\sigma_\Theta\right)$.
\end{proof}
A technical Lemma, proving that any time translation of a function in $\WTC(\R^M)$ is still a function in $\WTC(\R^M)$, is also necessary.
\begin{lem}\label{lem:flow-well-defined}
Let $\Theta$ be a suitable set of moduli of continuity. If $t\in\R$ and $f\in\WTC(\R^M)$, then $f_t\in\WTC(\R^M)$.
\end{lem}
\begin{proof}
Firstly, notice that for any fixed $f\in\WTC(\R^M)$ and $t\in\R$, the function $f_t$ trivially satisfies (C1) and (C2). In order to prove that condition (W) of Definition \ref{def:TCnWTC} holds, consider $j\in\N$, $I=[q_1,q_2], \ q_1,q_2\in\Q$, and  a sequence $\big(x_n(\cdot)\big)_{\nin}$ in $\K_j^I$   converging uniformly to $x(\cdot)\in\K_j^I$. Then, one has that
\begin{align}
\lim_{\nti}&\left|\int_I\big[f_t\big(s,x_n(s)\big)-f_t\big(s,x(s)\big)\big]\, ds\,\right|\nonumber\\
&\qquad\qquad=\lim_{\nti}\left|\int_I\big[f\big(s+t,x_n(s)\big)-f\big(s+t,x(s)\big)\big]\, ds\,\right|\nonumber\\
&\qquad\qquad=\lim_{\nti}\left|\int_{I+t}\big[f\big(u,x_n(u-t)\big)-f\big(u,x(u-t)\big)\big]\, du\,\right|.
\label{eq:27-06_18:49}
\end{align}
Considering an interval $J$ with rational extremes such that $I\cup(I+t)\subset J$ and, up to an extension by constants to $J$, the functions $x_n(\cdot\,-t)$ and $x(\cdot\,-t)$ are in $\K^J_j$. If $t\in\Q$ we immediately obtain the thesis thanks to Lemma \ref{lem:conv-subint}(i). If $t\in\R$, fix $\ep>0$ and let $\delta_1,\delta_2>0$ be such that
\begin{equation*}
\int_{q_1+t-\delta_1}^{q_1+t}m^j_f(u)\, du<\frac{\ep}{4}\qquad \text{and}\qquad \int_{q_2+t}^{q_2+t+\delta_2}m^j_f(u)\, du<\frac{\ep}{4}\,,
\end{equation*}
where $m^j_f(\cdot)$ is the $m$-bound of $f$ on $B_j$. The previous inequalities hold because of the continuity of the integral. Thus, denoted by $\delta=\min\{\delta_1,\delta_2\}$, consider $p_1\in[q_1+t-\delta,\, q_1+t]\cap\Q$ and $p_2\in[q_2+t,\, q_2+t+\delta]\cap\Q$. Starting from the last member of the chain of equalities in \eqref{eq:27-06_18:49}, one has
\begin{equation*}
\begin{split}
\lim_{\nti}\bigg|\int_{q_1+t}^{q_2+t}\big[&f\big(u,x_n(u-t)\big)-f\big(u,x(u-t)\big)\big]\, du\,\bigg|\\
&\le \lim_{\nti}\left|\int_{p_1}^{p_2}\big[f\big(u,x_n(u-t)\big)-f\big(u,x(u-t)\big)\big]\, du\,\right|\\
&\qquad\qquad+2\int_{p_1}^{q_1+t}m^j_f(u)\, du+2\int_{q_2+t}^{p_2}m^j_f(u)\, du\\
&\le\lim_{\nti}\left|\int_{p_1}^{p_2}\big[f\big(u,x_n(u-t)\big)-f\big(u,x(u-t)\big)\big]\, du\,\right|+\ep\,.
\end{split}
\end{equation*}
Therefore, we obtain the thesis, thanks to Lemma \ref{lem:conv-subint}(i), putting together the previous chain of inequalities and \eqref{eq:27-06_18:49}.
\end{proof}
\begin{prop}
Let $\Theta$ be a suitable set of moduli of continuity and $\sigma_\Theta$ the topology defined as in \rm{Definition \ref{def:TCnWTC}}. The following statements hold
\begin{itemize}
\item[(i)] If $E\subset \WTC(\R^M)$ (resp. $E\subset\LC(\R^M)$) admits $L^1_{loc}$-equicontinuous $m$-bounds  (resp. $L^1_{loc}$-equicontinuous  $l$-bounds), then $\mathrm{cls}_{(\WTC(\R^M),\sigma_\Theta)}(E)$ has $L^1_{loc}$-equicontinuous $m$-bounds  (resp.  $L^1_{loc}$-equicontinuous $l$-bounds).
\item[(ii)] If $E\subset \WTC(\R^M)$ (resp. $E\subset\LC(\R^M)$) admits $L^1_{loc}$-bounded $m$-bounds  (resp. $L^1_{loc}$-bounded $l$-bounds)  then $\mathrm{cls}_{(\WTC(\R^M),\sigma_\Theta)}(E)$ has $L^1_{loc}$-bounded $m$-bounds  (resp. $L^1_{loc}$-bounded $l$-bounds).
 \end{itemize}
 \label{prop:analog4.10paper}
\end{prop}
\begin{proof} Consider $E\subset \WTC(\R^M)$ with $L^1_{loc}$-equicontinuous $m$-bounds, that is, for every $j\in\N$ there is a family of $m$-bounds for $E$, namely $S^j=\{m^j_f(\cdot)\mid f\in E, \, m^j_f(\cdot) \text{ $m$-bound for } f \text{ on } B_j\}$, satisfying the condition in Definition \ref{weakcomp}. Moreover, we will assume, by simplicity, that for every $j\in\N$, $m_f^j(t)\le m_f^{j+1}(t)$ for almost every $t\in\R$.
Let us denote by $\overline E=\mathrm{cls}_{(\WTC(\R^M),\sigma_\Theta)}(E)$, and, for any $g\in \overline E$, let $(g_n)_\nin$ be a sequence in $E$ converging to $g$ in $\left(\WTC(\R^M),\sigma_\Theta\right)$.
\par
Now, consider the topological space $(\M^+,\widetilde{\sigma})$, i.e. the space of  positive and regular Borel measures on $\R$ with the topology $\widetilde \sigma$ defined through convergence of sequences as follows; we say that  a sequence  $(\mu_n)_\nin$ of measures in $\M^+$ vaguely converges to $\mu\in\M^+$, and write $\mu_n\xrightarrow[]{\widetilde{\sigma}}\mu$, if and only if
\begin{equation*}
 \lim_{\nti}\int_\R \phi(s)\, d\mu_n(s)=\int_\R\phi(s)\, d\mu(s)\qquad \text{for each } \phi\in C^+_C(\R).
\end{equation*}
Then, fixed $j\in\N$, for every $\nin$, let $\mu^j_n\in\M^+$ be the positive absolutely continuous measure (with respect to Lebesgue measure) with density $m^j_{g_n}(\cdot)$. We recall that since $S^j$ is $L^1_{loc}$-equicontinuous, then it is in particular $L^1_{loc}$-bounded, which implies that $\{\mu^j_n\mid\nin\}$ is relatively compact in $(\M^+,\widetilde{\sigma})$ (see Kallenberg \cite[Theorem 15.7.5, p.170]{book:Kall}). Thus $(\mu^j_n)_\nin$ vaguely converges, up to a subsequence, to a measure $\mu^j\in\M^+$.
Moreover, by Lebesgue-Besicovitch differentiation theorem, there exists $m^j(\cdot)\in L^1_{loc}$ such that
\begin{equation*}
m^j(t)=\lim_{h\to0}\frac{\mu^j([t,t+h])}{h}\, , \qquad \mathrm{for\ a.e.} \ t\in\R\, ,
\label{eq:16.06-11:32}
\end{equation*}
and $m^j(\cdot)$ is the density of the absolutely continuous part of the Radon-Nikod\'ym decomposition of $\mu^j$ in each compact interval. We claim that $m^j(\cdot)$ is an $m$-bound for $g$ on $B_j$. Let us fix $x\in B_j$, $I=[q_1,q_2]$ with $q_1,q_2\in\Q$, and firstly assume that $t,h\in\Q$, with $h>0$ such that $t+h\in I$. If $\phi\in C_C^+(\R)$ is such that $\phi\equiv 1$ in $[t,t+h]$, then recalling that $g_n\xrightarrow[]{\sigma_\Theta}g$ and using Lemma \ref{lem:conv-subint}, we have
\begin{equation*}
\begin{split}
\left|\,\frac{1}{h}\int_{t}^{t+h}g(s,x)\, ds\,\right|&=\lim_{\nti}\left|\,\frac{1}{h}\int_{t}^{t+h}g_n(s,x)\, ds\,\right|\\
&\le \lim_{\nti}\frac{1}{h}\int_\R\phi(s)\, m_{g_n}^j(s)\, ds=\frac{1}{h}\int_\R\phi(s)\, d\mu^j(s)\,.
\end{split}
\end{equation*}
Moreover, by the regularity of the measure $\mu^j$ one has that
\begin{equation*}
 \mu^j\big([t,t+h]\big)=\inf\left\{\int_\R\phi(s)\,  d\mu^j(s)\;\Big|\; \phi\in C_C^+(\R),\;  \phi\equiv 1\; \text{in } [t,t+h]\right\}\,.
\end{equation*}
Hence, for any $t,h\in\Q$, with $h>0$, one has
\begin{equation}
\left|\,\frac{1}{h}\int_{t}^{t+h}g(s,x)\, ds\,\right|\le\frac{\mu^j([t,t+h])}{h}\,.
\label{eq:27-06_12:23}
\end{equation}
Now, consider $t,h\in\R$, with $h>0$ and $t+h\in I$, and let $(s_n)_\nin$ and $(t_n)_\nin$ be two sequences in $\Q$ such that, as $\nti$, $s_n\downarrow t$ and $t_n\uparrow t+h$, respectively. By \eqref{eq:27-06_12:23}, applied on the intervals $[s_n,t_n]$, and noticing that $\mu^j([s_n,t_n)]\le \mu^j([t,t+h])$ for every $\nin$, one can write
\begin{equation}
\left|\,\frac{1}{h}\int_{s_n}^{t_n}g(s,x)\, ds\,\right|\le\frac{\mu^j([t,t+h])}{h}\,,\quad\text{for all }\nin\,.
\label{eq:13-07_14:05}
\end{equation}
Therefore, passing to the limit as $\nti$ and using the continuity of the integral, one obtains \eqref{eq:27-06_12:23} for every $t,h\in\R$ with $h>0$ and $t+h\in I$.
As a further step, we take the limit as $h\to0$ and obtain that for every $x\in B_j$ there exists $I_x\subset I$ with $\meas(I\setminus I_x)=0$ such that for all $t \in I_x$ one has
\begin{equation}
|g(t,x)|\le m^j(t)\,.
\label{eq:16.06-11:45}
\end{equation}
From the arbitrariness of $I$, for any fixed $x\in B_j$ one obtains \eqref{eq:16.06-11:45} for almost every $t\in\R$ using a numerable covering of the real line. For every fixed $x\in B_j$ let us now denote, with a little bit of abuse of notation, by $I_x$ the subset of $\R$ such that $\meas(\R\setminus I_x)=0$ and \eqref{eq:16.06-11:45} holds for all $t\in I_x$. Such a set clearly depends on $x\in B_j$. However, by Fubini's Theorem we obtain that for almost every $t\in \R$ the inequality \eqref{eq:16.06-11:45} holds for almost every $x\in B_j$. Therefore, we look for a new function $g^*\in\WTC(\R^M)$ that coincides with $g$ almost everywhere, which implies that $g$ and $g^*$ are in fact representatives of the same element in $\WTC(\R^M)$,  and such that for almost every $t\in\R$, the function $g^*$ satisfies an inequality of the type \eqref{eq:16.06-11:45} for all $x\in B_j$. Let us consider the function $g^*\colon\R\times\R^N\to\R^M$ defined as follows: for every $ t\in\R$ we set
\begin{equation}
g^*(t,x)=
\begin{cases}
g(t,x)&\text{if }x\in B_i\setminus B_{i-1}\text{ and }|g(t,x)|\le m^{i+1}(t)\,,\text{ with } i\in\N\\
\mathbf{0}&\text{otherwise}\,,
\end{cases}
\label{eq:19/06-18:45}
\end{equation}
where $\mathbf{0}$ represents the zero vector of $\R^N$. The function $g^*$ is Borel measurable and coincides with $g$ almost everywhere. Furthermore, we have that for each $j\in\N$ and for every $t\in\R$, $g^*$ satisfies
\begin{equation}
|g^*(t,x)|\le m^{j+1}(t)\,,
\label{eq:13/06-14:41}
\end{equation}
for all $x\in B_j$. Thus, $g^*$ satisfies (C1) and (C2). Therefore, to prove that $g$ and $g^*$ are representatives of the same element in $\WTC(\R^M)$, we only need to prove that $g^*$ satisfies (W) of \ref{def:TCnWTC}. In order to do that, we firstly show that for any $I=[q_1,q_2]$, $q_1,q_2\in\Q$, if $x(\cdot)\in\K^I_j$, then $g^*\big(t,x(t)\big)=g\big(t,x(t)\big)$ for almost every $t\in I$. Let $x(\cdot)\in\K^I_j$ and reason locally. Consider $t_0\in I$ and assume that $i\le|x(t_0)|<i+1$ for some $i\in\N$. Then, by the continuity of $x(\cdot)$, there exist $\delta>0$, such that  $|x(t)|\in(i-1,i+1]$ for every $t\in I_{t_0}=[t_0-\delta,t_0+\delta]\cap I$. Let $\widetilde x(\cdot)$ be the continuous function defined on $I$ which coincides with $x(\cdot)$ on $I_{t_0}$ and it is its extension by constants on $I\setminus I_{t_0}$. Trivially, $\widetilde x(\cdot)\in\K^I_j$ and $\|\widetilde x(\cdot)\|_{L^\infty(I_{t_0})}\le i+1$. Hence, for every $t\in I\cap\Q$ and for every $h\in\Q$, with $h>0$ and $t+h\in I$, considered $\phi\in C_C^+(\R)$ such that $\phi\equiv 1$ in $[t,t+h]$ and using Lemma \ref{lem:conv-subint}, we have that
\begin{equation*}
\begin{split}
\left|\,\frac{1}{h}\int_{t}^{t+h}g\big(s,\widetilde x(s)\big)\, ds\,\right|&=\lim_\nti\left|\,\frac{1}{h}\int_{t}^{t+h}g_n\big(s,\widetilde x(s)\big)\, ds\,\right| \\
&\le\lim_\nti\frac{1}{h}\int_\R\phi(s)\, m^{i+1}_{g_n}(s)\, ds=\frac{1}{h}\int_\R\phi(s)\,d\mu^{i+1}(s)\,.
\end{split}
\end{equation*}
Reasoning as in \eqref{eq:13-07_14:05}, one can prove that the previous inequality actually holds for any $t\in I$ and $h>0$ such that $t+h\in I$. Thus, taking the limit as $h\to 0$ and reasoning as before, we obtain that  $\big|g\big(t,\widetilde x(t)\big)\big|\le m^{i+1}(t)$ for almost every $t\in I$.  In particular,  for almost every $t\in  I_{t_0}$,
\[ \big|g\big(t, x(t)\big)\big|\le m^{i+1}(t) \]
and recalling how $g^*$ is defined in \eqref{eq:19/06-18:45}, we have that $g^*\big(t,x(t)\big)=g\big(t,x(t)\big)$ for  almost every $t\in  I_{t_0}$. Thanks to the compactness of $I$, we can repeat such an argument a finite number of times and deduce that actually $g^*\big(t,x(t)\big)=g\big(t,x(t)\big)$ for  almost every $t\in I$. As a consequence, one can easily prove that condition (W) of Definition \ref{def:TCnWTC} holds for $g^*$.
Therefore, $g$ and $g^*$ are two representatives of the same element of $\WTC(\R^M)$, because both are in $\WTC(\R^M)$ and only differ from each other on a negligible subset of $\R\times\R^N$.
\par\smallskip

Finally, we prove that $\overline E$ admits $L^1_{loc}$-equicontinuous $m$-bounds. For each $f\in \overline E$ and any $j\in\N$ let $m_f^j$ be either, the $m$-bound of $f$ in $S^j$ if $f\in E$, or the $m$-bound given by \eqref{eq:13/06-14:41} if $f\in \overline E\setminus E$, i.e. the absolutely continuous part of a limit measure.
Consider $j\in\N$, $r,\ep>0$ and  let $\delta=\delta(r,\ep)>0$ be the one given by the $L^1_{loc}$-equicontinuity of $S^{j+1}$. If $t,s\in[-r,r]$ with $s<t$, $t-s<\delta/3$, and $\phi\in C^+_C$ is such that $\supp \phi\subset [s-\delta/3,t+\delta/3]$ and $\phi\equiv 1$ in $[s,t]$. Then, we have
\begin{equation*}
\begin{split}
\int_{s}^t m_f^j(u)\, du&\le\int_\R\phi(u)\,  m_f^j(u)\, du\le \lim_\nti \int_\R \phi(u)\, m_{f_n}^{j+1}(u)\, du\\
&\le\sup_{g\in E}\int_{s-\delta/3}^{t+\delta/3} \, m_{g}^{j+1}(u)\, du<\ep\, ,
\end{split}
\end{equation*}
and thus, taking the superior over the  functions in $\overline E$ in the previous expression, one gets\par\vspace{-.69cm}
\begin{equation*}
\sup_{f\in\overline E}\int_{s}^t m_f^j(u)\, du<\ep\, .
\end{equation*}
Therefore, $\overline E$ admits $L^1_{loc}$-equicontinuous $m$-bounds. Analogous reasonings apply to the remaining cases in (i) and (ii).
\end{proof}
Below, the definition of hull of a function is given.
\begin{defn}\label{def:Hull}
Let $E$ denote one of the spaces in~\eqref{eq:SPincl} and $\sigma$  one of the topologies in \eqref{eq:TopIncl}, assuming that endowing $E$ with the topology $\sigma$ makes sense. If $f\in E$, we call  \emph{the hull of $f$ with respect to $(E,\sigma)$}, the topological subspace of  $(E,\sigma)$ defined~by
\begin{equation*}
\mathrm{Hull}_{(E,\sigma)}(f)=\big(\mathrm{cls}_{(E,\sigma)}\{f_t\mid t\in\R\} ,\, \sigma\big) ,
\end{equation*}
where, $\mathrm{cls}_{(E,\sigma)}(A)$ represents the closure in $(E,\sigma)$ of the set $A$, and $\sigma$ is the induced topology.
\end{defn}
Thus, as a corollary of Proposition \ref{prop:analog4.10paper} and considering the previous definition, one has the following result.
\begin{cor}
Let $\Theta$ be a suitable set of moduli of continuity and $\sigma_\Theta$ the topology defined as in \rm{Definition \ref{def:TCnWTC}}. The following statements hold
\begin{itemize}
\item[(i)] If $f\in \WTC(\R^M)$ (resp. $\LC(\R^M)$) has $L^1_{loc}$-equicontinuous $m$-bounds  (resp. $L^1_{loc}$-equicontinuous  $l$-bounds), then any $g\in\mathrm{Hull}_{(\WTC(\R^M),\sigma_\Theta)}(f)$  has $L^1_{loc}$-equicontinuous $m$-bounds  (resp.  $L^1_{loc}$-equicontinuous $l$-bounds).
\item[(ii)]
 If $f\in \WTC(\R^M)$ (resp. $\LC(\R^M)$) has $L^1_{loc}$-bounded $m$-bounds  (resp. $L^1_{loc}$-bounded $l$-bounds)  then any $g\in\mathrm{Hull}_{(\WTC(\R^M),\sigma_\Theta)}(f)$ has $L^1_{loc}$-bounded $m$\nbd-bounds  (resp. $L^1_{loc}$-bounded $l$-bounds).
 \end{itemize}
\label{cor:17-05_13:35}
\end{cor}
\section{Continuity of the flow}
\label{sec:cont-Flow}
This section contains several results of continuity for skew-product flows generated by either, a singular Carath\'eodory system with vector field in $\LC$ or,  by triangular systems composed of a nonlinear system with vector field in $\LC$ and a linear system with vector field in $\WTC$. The second case assumes additional relevance when the linear system is the variational equation of the non-linear one. The classic theory of Carath\'eodory ODEs  provides the differentiability of the solutions with respect to the initial conditions when the respective vector fields are continuously differentiable with respect to $x$ (see Kurzweil \cite{book:Kurz}). Nevertheless, in the last part of the section, we provide conditions that allow to extend such conclusions to the solutions of specific Carath\'eodory differential equations whose vector fields may possibly not admit continuous partial derivatives with respect to $x$. In particular, we recall the notion of linearized skew-product flow given in \cite[Definition 6.2]{paper:LNO} and introduce a weak and new type of continuous linearized skew-product flow.
\par \smallskip
We start with a result of continuity of the base flow, that is, continuity of the time translations in the space $\left(\WTC(\R^M),\sigma_{\Theta}\right)$.
\begin{thm}\label{thm:theta_translCont}
Let $\Theta$ be a suitable set of moduli of continuity and consider the space $\left(\WTC(\R^M),\sigma_{\Theta}\right)$. If $E\subset\WTC(\R^M)$ admit $L^1_{loc}$-equicontinuous $m$-bounds, then, denoted by $\overline E=\mathrm{cls}_{(\WTC(\R^M),\sigma_\Theta)}(E)$, one has that the map
\begin{equation*}
\varphi:\R\times \overline E\to \WTC(\R^M)\, ,\qquad (t,f)\mapsto\varphi(t,f)=f_t\, ,
\end{equation*}
is  well-defined and continuous.
\label{thm:continuityBASE}
\end{thm}
\begin{proof}
Firstly, notice that the map is well-defined thanks to Lemma \ref{lem:flow-well-defined}. Let $(f_n)_{\nin}$  be a sequence in $E$  converging to $f$ in $\left(\WTC(\R^M),\sigma_{\Theta}\right)$ and $(t_n)_\nin$ a sequence in $\R$ converging to  $t\in\R$. We want to prove that for every $I=[q_1,q_2]$, $q_1,q_2\in\Q$, and every $j\in\N$ one has that
\begin{equation*}
\lim_\nti \sup_{x(\cdot)\in\K^I_j} \bigg|\int_I\big[f_n\big(t_n+s,x(s)\big)-f\big(t+s,x(s)\big)\big]ds\,\bigg|=0\,.
\end{equation*}
Let us fix $\ep>0$, $j\in\N$, and $I=[q_1,q_2]$, $q_1,q_2\in\Q$ and consider an interval $[r_1,r_2]$ such that, for every $\nin$, one has $[q_1+t_n, q_2+t_n]\subset[r_1,r_2]$.  Since $E$ admits $L^1_{loc}$-equicontinuous $m$-bounds, and thanks to Proposition \ref{prop:analog4.10paper}, one has that there exists $\delta>0$ such that
\begin{equation*}
\sup_{g\in \overline E}\int_{\tau_1}^{\tau_2} m_g^j(u)\,du<\ep/6\,,
\end{equation*}
whenever $\tau_1,\tau_2\in[r_1,r_2]$ and $0< \tau_2-\tau_1<\delta$. Consider $p_1(t),\, p_2(t)\in\Q$ such that $q_1+t<p_1(t)<p_2(t)<q_2+t$ and
\begin{equation*}
 p_1(t)-q_1-t<\delta\quad\text{and}\quad q_2+t-p_2(t)<\delta\,.
\end{equation*}
Notice also that, since $t_n\to t$, then there exists $n_0\in\N$ such that for every $n>n_0$ one has that $q_1+t_n<p_1(t)<p_2(t)<q_2+t_n$ and
\begin{equation*}
 p_1(t)-q_1-t_n<\delta\quad\text{and}\quad q_2+t_n-p_2(t)<\delta\,.
\end{equation*}
Then, for every $n>n_0$ one has that
\begin{equation*}
\begin{split}
&\sup_{x(\cdot)\in\K^I_j} \bigg|\int_I\big[f_n\big(t_n+s,x(s)\big)-f\big(t+s,x(s)\big)\big]ds\,\bigg|\\
&\;\quad=\sup_{x(\cdot)\in\K^I_j} \bigg|\int_{q_1+t_n}^{q_2+t_n}f_n\big(u,x(u-t_n)\big)du- \int_{q_1+t}^{q_2+t}f\big(u,x(u-t)\big)du\,\bigg|\\
&\;\quad\le\sup_{x(\cdot)\in\K^I_j} \bigg|\int_{p_1(t)}^{p_2(t)}\big[f_n\big(u,x(u-t_n)\big)-f\big(u,x(u-t)\big)\big]du\,\bigg|+\frac{4\ep}{6}\\
&\;\quad\le\sup_{x(\cdot)\in\K^I_j} \bigg|\int_{p_1(t)}^{p_2(t)}\big[f_n\big(u,x(u-t_n)\big)-f\big(u,x(u-t_n)\big)\big]du\,\bigg|+\\
&\qquad+\sup_{x(\cdot)\in\K^I_j} \bigg|\int_{p_1(t)}^{p_2(t)}\!\!\big[f\big(u,x(u-t_n)\big)-f\big(u,x(u-t)\big)\big]du\,\bigg|+\frac{2\ep}{3} = P_n+ R_n+\frac{2\ep}{3}.
\end{split}
\end{equation*}
Then, if we take an interval $J$ with rational extremes such that $I\cup [p_1(t),p_2(t)]\subset J$,  up to a suitable extension by constants to $J$, the functions  $y_n(\cdot)=x(\cdot\,-t_n)$  belong to $\K_j^J$ and we deduce that
\[ \lim_{n\to\infty}P_n\le \lim_{\nti}\sup_{y(\cdot)\in\K^J_j} \bigg|\int_{p_1(t)}^{p_2(t)}\big[f_n\big(u,y(u)\big)-f\big(u,y(u)\big)\big]\,du\bigg|=0\]
because $(f_n)_{\nin}$ converges to $f$ in $\left(\WTC(\R^M),\sigma_{\Theta}\right)$ and thanks to Lemma~\ref{lem:conv-subint}(ii).\par
Analogously, recalling that $f\in\WTC$  satisfies an equality of the type \eqref{WT}, from  Lemma~\ref{lem:conv-subint}(i) we deduce that $\lim_{n\to\infty}R_n=0$, which finishes the proof.
\end{proof}
As a corollary of the previous theorem, one has that the following map is well defined and continuous. Thus,  a continuous flow on the hull of a function in $\WTC$ with $L^1_{loc}$-equicontinuous $m$-bounds is obtained.
\begin{cor}
Let $f\in\WTC(\R^M)$  admit $L^1_{loc}$-equicontinuous $m$-bounds. Then, the map
\begin{equation*}
\varphi:\R\times \mathrm{Hull}_{(\WTC,\sigma_\Theta)}(f)\to \mathrm{Hull}_{(\WTC,\sigma_\Theta)}(f)\, ,\qquad (t,g)\mapsto\varphi(t,g)=g_t\, ,
\end{equation*}
defines a continuous flow on $\mathrm{Hull}_{(\WTC,\sigma_\Theta)}(f)$.
\end{cor}
As follows, ordinary differential equations whose vector fields belong to the already introduced Carath\'eodory spaces, are treated. For the sake of completeness and to set some notation, we state a theorem of existence and uniqueness of the solution for Cauchy problems of Carath\'eodory type. A proof can be found in Coddington and Levinson \cite[Theorems 1.1, 1.2 and 2.1]{book:CL}.
\begin{thm}
For any $f\in\LC$ and any $x_0\in\R^N$ there exists a maximal interval $I_{f,x_0}=(a_{f,x_0},b_{f,x_0})$ and a unique continuous function $x(\cdot,f,x_0)$ defined on $I_{f,x_0}$ which is the solution of the Cauchy problem
\begin{equation*}\label{eq:solLCE}
\dot x=f(t,x)\, ,\qquad x(0)=x_0\, .
\end{equation*}
In particular, if $a_{f,x_0}>-\infty$ (resp. $b_{f,x_0}<\infty$), then $|x(t,f,x_0)|\to\infty$ as $t\to a_{f,x_0}$ (resp. as $t\to b_{f,x_0}$).
\label{thm:05.07-13:44}
\end{thm}
\begin{cor}\label{cor:14:07-21:05}
Let $\Theta$ be a suitable set of moduli of continuity. For any $f\in\LC$, $F\in\WTC\big(\R^{N\times N}\big),\, h\in\WTC$, and $x_0,\, y_0\in\R^N$, there exists a unique solution of the Cauchy problem
\begin{equation}
\begin{cases}
\dot x=f(t,x), &x(0)=x_0\,,\\
\dot y=F(t,x)\,y+h(t,x), &y(0)= y_0\,,
\end{cases}
\label{eq:system}
\end{equation}
which will be denoted by $(x(\cdot,f,x_0),y(\cdot,f,F,h,x_0, y_0))$, and whose maximal interval of definition coincides with the interval $I_{f,x_0}$ provided by {\rm Theorem~\ref{thm:05.07-13:44}}.
\end{cor}
\begin{defn}\label{def:THETAjseq}
Let $E\subset\LC$ admit $L^1_{loc}$-equicontinuous $m$-bounds. For any $j\in\N$ and for any interval $I=[q_1,q_2]$, $q_1,q_2\in\Q$, define
\begin{equation*}
\theta^I_j(s):=
 \sup_{t\in I,f\in E}\int_t^{t+s}m_f^j(u)\, du\, ,
\end{equation*}
where, for any $f\in E$, the function $m_f^j(\cdot)\in L^1_{loc}$ denotes the optimal $m$-bounds of $f$ on $B_j$. Notice that, since $E$ admits $L^1_{loc}$-equicontinuous $m$-bounds, then  $\Theta=\{\theta^I_j(\cdot)\mid I=[q_1,q_2],\, q_1,q_2\in\Q,\, j\in\N\}$ defines a suitable set of moduli of continuity.
\end{defn}
\begin{rmk}\label{rmk:THETAjfunc}
If $f\in\LC$ has $L^1_{loc}$-equicontinuous $m$-bounds we similarly define for any $B_j\subset\R^N$,
\begin{equation*}
\theta_j(s):= \sup_{t\in\R}\int_t^{t+s}m^j(u)\, du\, ,
\end{equation*}
where $m^j(\cdot)$ is the optimal $m$-bound for $f$ on $B_j$. Here again, notice that $\Theta=\{\theta^I_j(\cdot)\mid I=[q_1,q_2],\, q_1,q_2\in\Q,\, j\in\N\}$  defines a suitable set of moduli of continuity thanks to the $L^1_{loc}$-equicontinuity.
\end{rmk}
Eventually, before proving the theorem of continuity of the solutions with respect to the initial data and the variation on the vector fields, we present a technical lemma that will be necessary in the second part of the proof of the theorem. The proof is carried out through standard arguments of measure theory and thus it is omitted.
\begin{lem}
Let $\big(a_n(\cdot)\big)_\nin$ be a sequence in $L^1_{loc}$ such that $\{a_n(\cdot)\mid\nin\}$ is $L^1_{loc}$-bounded. If there exists $a(\cdot)\in L^1_{loc}$ such that  for any $ t_1,t_2\in\Q$ with $t_1<t_2$ one has
\begin{equation}
 \lim_\nti\int_{t_1}^{t_2}a_n(s)\, ds= \int_{t_1}^{t_2}a(s)\, ds\,,
\label{eq:02.05-12:12}
\end{equation}
then, for any $ t_1,t_2\in\Q$
\begin{equation}
\lim_\nti\int_{t_1}^{t_2}a_n(s)\phi_n(s)\, ds= \int_{t_1}^{t_2}a(s)\phi(s)\, ds\,,
\label{eq:02.05-11:43}
\end{equation}
where $\big(\phi_n(\cdot)\big)_\nin$ is any sequence in $C[t_1,t_2]$ converging uniformly to some $\phi(\cdot)\in C[t_1,t_2]$.
\label{lem:25.04}
\end{lem}
\begin{thm}\label{thm:ContODETthetaE}
Consider $E\subset\LC$  with $L^1_{loc}$-equicontinuous $m$-bounds and let $\Theta=\{\theta^I_j\mid I=[q_1,q_2],\,  q_1,q_2\in\Q,\,  j\in\N\}$ be the countable family of moduli of continuity in {\rm Definition~\ref{def:THETAjseq}}. Additionally, consider $B\subset\WTC\big(\R^{N\times N}\big)$ and $C\subset\WTC$, both with $L^1_{loc}$-equicontinuous $m$-bounds. With the notation of {\rm Theorem~\ref{thm:05.07-13:44}} and {\rm Corollary~\ref{cor:14:07-21:05}},
\begin{itemize}
\item[(i)] if $(f_n)_\nin$  in $E$ converges to $f$ in $(\LC,\sigma_\Theta)$ and $(x_{0,n})_\nin$ in $\R^N$ converges to $x_0\in\R^N$, then
    \[ \qquad x(\cdot,f_n,x_{0,n}) \xrightarrow{\nti}  x(\cdot,f,x_0)\]
     uniformly in any $[T_1,T_2]\subset I_{f,x_0}$;\par\vspace{0.05cm}
\item[(ii)] if $(F_n)_\nin$ in $B$ converges to $F$ in $\big(\WTC\big(\R^{N\times N}\big),\sigma_\Theta\big)$, $(h_n)_\nin$ in $C$ converges to $h$ in $(\WTC,\sigma_\Theta)$, and  $( y_{0,n})_\nin$ in $\R^N$ converges to $ y_0\in\R^N$, then
    \[  \hspace{1cm} y(\cdot,f_n,F_n,h_n,x_{0,n}, y_{0,n})\xrightarrow{\nti} y(\cdot,f,F,h,x_0,  y_0)\]
uniformly  in any $[T_1,T_2]\subset I_{f,x_0}$.
\end{itemize}
\end{thm}
\begin{proof}
(i) We will prove the uniform convergence of $\big(x(\cdot,f_n,x_{0,n})\big)_\nin$ to  $x(\cdot,f,x_0)$ in $[0,T]$ for any $0< T<b_{f,x_0}$. The case $a_{f,x_0}<T<0$ is analogous. Denote
\begin{equation}
0<\rho =1+ \max\big\{ (|x_{0,n}|)_\nin,\  \|x(\cdot,f,x_0)\|_{L^\infty([0,T])}\big\}\, ,
\label{eq:DefRHO}
\end{equation}
and define
\begin{equation*}
z_n(t)=
\begin{cases}
x(t,f_n,x_{0,n}), & \text{if $0\le t< T_n$,} \\
x(T_n,f_n,x_{0,n}), & \text{if $T_n\le t\le T$.}
\end{cases}
\end{equation*}
where $T_n=\sup\{t\in[0,T]\mid|x(s,f_n,x_{0,n})|\le \rho,\, \forall\, s\in[0,t]\}$. Notice that by \eqref{eq:DefRHO} and by the continuity of $\big(x(\cdot,f_n,x_{0,n})\big)_\nin$, we have that $T_n>0$ for any $n\in\N$. In particular notice that $(z_n(\cdot))_{\nin}$ is uniformly bounded. Moreover, consider $j\in\N$ so that $\rho<j$ and let $(m_n(\cdot))_{\nin}=(m_{f_n}^j(\cdot))_{\nin}$ be the sequence of optimal $m$-bounds of $(f_n)_{\nin}$ on $B_j$. If $t_1,t_2\in[0,T_n)$, $t_1<t_2$, then
\begin{equation}
|z_n(t_1)-z_n(t_2)|\le \int_{t_1}^{t_2}\big|f_n\big(s,z_n(s)\big)\big|\, ds \le \int_{t_1}^{t_2}m_n(s)\, ds\, .
\label{eq:14.06-18:11}
\end{equation}
Fixed $\ep>0$,  since $E$ admits $L^1_{loc}$-equicontinuous $m$-bounds, there exists $\delta=\delta(T,\ep)>0$ such that, if $0\le t_1 \le t_2< T_n$, then the right-hand side in \eqref{eq:14.06-18:11} is smaller than $\ep$ whenever $t_2-t_1<\delta$. Notice that, in fact, the inequality $|z_n(t_1)-z_n(t_2)|<\ep$ is true on the whole interval $[0,T]$ whenever $t_2-t_1<\delta$ because in $[T_n,T]$ the difference on the left side of equation \eqref{eq:14.06-18:11} is zero. Thus, the sequence  $(z_n(\cdot))_{\nin}$ is equicontinuous. Then, Ascoli-Arzel\'a's theorem implies that $(z_n(\cdot))_{\nin}$ converges uniformly, up to a subsequence, to some continuous function $z:[0,T]\to \R^N$.
\par\vspace{0.05cm}
In order to conclude the proof, we prove that $z(\cdot)\equiv x(\cdot,f,x_0)$ in $[0,T]$. Define
\begin{equation}
T_0=\sup \{t\in[0,T]\mid |z(s)|<\rho-1/2\quad \forall \, s\in[0,t]\}\, ,
\label{eq:defT0}
\end{equation}
and notice that $T_0>0$ because  $(x_{0,n})_\nin$ converges to $x_0$ and $z(\cdot)$ is continuous.  Since $z_n(\cdot)$ converges uniformly to $z(\cdot)$ in $[0,T]$, then there exists $n_0\in\N$ such that if $n>n_0$, then
\begin{equation*}
|z_n(t)|<\rho-1/4\qquad \forall \, t\in[0,T_0]\, .
\end{equation*}
Therefore, for any $t\in[0,T_0]$  and for any $n>n_0$ one has  $z_n(t)=x(t,f_n,x_{0,n})$ and thus
\begin{equation}
z_n(t)=x_{0,n}+\int_0^tf_n\big(s,z_n(s)\big)\, ds\, ,\qquad t\in[0,T_0]\, ,\ n>n_0\, .
\label{eq:14.06-19:35}
\end{equation}
Now let us fix $ t\in[0,T_0]\cap \Q$ and consider the compact set $\mathcal{K}=\{z_n(\cdot)\mid\nin\}\cup\{z(\cdot)\}\subset C\big([0,t],\R^N\big)$. Notice that $\mathcal{K}\subset\mathcal{K}^{[0,t]}_j$ for the previously identified $j\in\N$. Moreover, remind that $(f_n)_{\nin}$ converges to $f$ in $\sigma_\Theta$, $(z_n(\cdot))_{\nin}$ converges uniformly to $z(\cdot)$ in $[0,T]$ and $(x_{0,n})_{\nin}$ converges to $x_0$ as $\nti$. Then, passing to the limit in \eqref{eq:14.06-19:35}, we have that
\begin{equation*}
z(t)=x_0+\int_0^tf\big(s,z(s)\big)\, ds\qquad \text{for }  t\in[0,T_0]\cap \Q\, .
\end{equation*}
As a matter of fact, the equality holds on the whole interval $[0,T_0]$ because of the continuity of $z(\cdot)$ and of the integral operator. Therefore, $z(\cdot)$ coincides with $x(\cdot,f,x_0)$ on  $[0,T_0]$. We prove that $T_0=T$ in order to conclude the proof. Otherwise, by \eqref{eq:defT0} and by the continuity of $z(\cdot)$, one would have $|z(T_0)|=|x(T_0,f,x_0)|=\rho-1/2$, which contradicts \eqref{eq:DefRHO}. Hence, $T_0=T$, as claimed, and thus for any  $t\in [0,T]$ we have that $x(t,f,x_0)= z(t)$ and $x(t,f_n,x_{0,n})= z_n(t)$ for any $\nin$, which concludes the proof of (i).\par\smallskip
(ii) In order to simplify the notation, let us denote by $x_n(\cdot)=x(\cdot,f_n,x_{0,n})$,  $y_n(\cdot)=y(\cdot,f_n,F_n,h_n,x_{0,n}, y_{0,n})$, $x(\cdot)=x(\cdot,f,x_0)$, and $y(\cdot)=y(\cdot,f,F,h,x_0,  y_0)$. Consider $0< T<b_{f,x_0}$ and, as well as we did in \eqref{eq:DefRHO} of  (i), define $0<\rho =1+ \max\big\{ (|y_{0,n}|)_\nin,\  \|y(\cdot)\|_{L^\infty([0,T])}\big\}$. Define the functions
\begin{equation*}
\zeta_n(t)=
\begin{cases}
y_n(t), & \text{if $0\le t< T_n$,} \\
y_n(T_n), & \text{if $T_n\le t\le T$.}
\end{cases}
\end{equation*}
where $T_n=\sup\{t\in[0,T]\mid|y_n(s)|\le \rho,\, \forall\, s\in[0,t]\}$. Then, thanks to the assumptions on  $B\subset\WTC\big(\R^{N\times N}\big)$ and $C\subset\WTC$ and reasoning like in \eqref{eq:14.06-18:11} one can easily prove that the sequence $\big(\zeta_n(\cdot)\big)_\nin$ is equicontinuous. Thus, since $\big(\zeta_n(\cdot)\big)_\nin$ is also uniformly bounded by construction, once again we obtain that it converges uniformly up to a subsequence to some $\zeta(\cdot)\in C\big([0,T]\big)$. Therefore, defining $T_0$ similarly to \eqref{eq:defT0}, we have that there exists $n_0\in\N$ such that if $n>n_0$, then for any $t\in [0,T_0]$ one has that
\begin{equation*}
\zeta_n(t)=y_{0,n}+\int_0^tF_n\big(s,x_n(s)\big)\zeta_n(s)\, ds+\int_0^th_n\big(s,x_n(s)\big)\, .
\end{equation*}
Hence, using the fact that  $(h_n)_\nin$ converges to $h$ in $(\WTC,\sigma_\Theta)$, $( y_{0,n})_\nin$  converges to $ y_0\in\R^N$, $(F_n)_\nin$ converges to $F$ in $\big(\WTC\big(\R^{N\times N}\big),\sigma_\Theta\big)$ and applying Lemma \ref{lem:25.04} with $a_n(t)=F_n\big(t,x_n(t)\big)$, $a(t)=F\big(t,x(t)\big)$, $\phi_n(t)=\zeta_n(t)$ and $\phi(t)=\zeta(t)$, one has that passing to the limit as $\nti$
\begin{equation*}
\zeta(t)=y_{0}+\int_0^tF\big(s,x_n(s)\big)\zeta(s)\, ds+\int_0^th\big(s,x(s)\big)\qquad \text{for }  t\in[0,T_0]\cap \Q\, .
\end{equation*}
Reasoning as in the last part of (i), one obtains the previous inequality on the whole interval $[0,T_0]$ and eventually proves that $T_0=T$, which ends the proof.
\end{proof}
Consider $f\in\LC$, a suitable set of moduli of continuity $\Theta=(\theta_j)_{j\in\N}$, and the family of differential equations $\dot x =g(t,x)$, where $g\in\mathrm{Hull}_{(\LC,\sigma_\Theta)}(f)$. With the notation introduced in Theorem \ref{thm:05.07-13:44}, let us denote by $\U_1$ the subset of $\R\times\mathrm{Hull}_{(\LC,\sigma_\Theta)}(f)\times\R^N$ given by
\begin{equation*}
\U_1=\bigcup_{\substack{g\in\mathrm{Hull}_{(\LC,\sigma_\Theta)}(f)\,,\\x\in\R^N}}
\{(t,g,x)\mid t\in I_{g,x}\}\,.
\end{equation*}
Analogously, let $f\in\LC$, $F\in \WTC(\R^{N\times N})$ and $h\in \WTC$, where $\Theta=(\theta_j)_{j\in\N}$ is a suitable set of moduli of continuity, and consider the family of differential equations of the type~\eqref{eq:system} for $(g,G,k)\in \Hu=\mathrm{Hull}_{\left(\LC\times \WTC\times\WTC,\sigma_\Theta\times\sigma_\Theta\times\sigma_\Theta\right)}(f,F,h)$, where the hull is constructed as in Definition~\ref{def:Hull}. Then, we denote by $\U_2$ the subset of $\R\times\Hu\times\R^N\!\!\times\R^N$ given by
\begin{equation*}
\U_2=\bigcup_{\substack{(g,G,k)\in\Hu\,,\\ x_0\in\R^N}}\{(t,g,G,k,x_0,y_0)\mid t\in I_{g,x_0}\,,\; y_0\in\R^n\}\,.
\end{equation*}
With the previous notation we can state the following theorem.
\begin{thm}
Let the functions $f\in\LC$, $F\in \WTC(\R^{N\times N})$ and $h\in \WTC$ have $L^1_{loc}$-equicontinuous $m$-bounds, where $\Theta=(\theta_j)_{j\in\N}$ is the suitable set of moduly of continuity given by the $m$-bounds of $f$ as shown in {\rm Remark \ref{rmk:THETAjfunc}}.
\begin{itemize}
\item[(i)] The set $\U_1$ is open in $\R\times\mathrm{Hull}_{(\LC,\sigma_\Theta)}(f)\times\R^N$ and the map
\begin{equation*}
\begin{split}
\Pi\colon  \ \U_1\subset \R\times\mathrm{Hull}_{(\LC,\sigma_\Theta)}(f)\times\R^N\  &\to\ \ \mathrm{Hull}_{(\LC,\sigma_\Theta)}(f)\times\R^N\\
\ (t,g,x_0)\quad \qquad \qquad &\mapsto\qquad \big(g_t, x(t,g,x_0)\big)
\end{split}
\end{equation*}
defines a local continuous skew-product flow on $\mathrm{Hull}_{(\LC,\sigma_\Theta)}(f)\times\R^N$.
\item[(ii)] The set $\U_2$ is open in $\R\times \Hu\times \R^N\times
\R^N$ and the map
\begin{equation*}
\begin{split}
\qquad\quad\;\;\Psi\colon \ \U_2\subset \R\times\Hu\times \R^N\!\!\times \R^N &\to\qquad\qquad\qquad\Hu\times \R^N \!\!\times \R^N\\
 (t,g,G,k,x_0, y_0) \  \ & \mapsto\;  \big(g_t, G_t, k_t ,x(t,g,x_0),y(t,g,G,k,x_0, y_0)\big)
\end{split}
\end{equation*}
defines a local continuous skew-product flow on $\Hu\times \R^N\!\!\times \R^N$.
\end{itemize}
\label{thm:contFlowTtheta}
\end{thm}
\begin{proof}
The proof is a direct consequence of Theorem \ref{thm:continuityBASE}, Theorem \ref{thm:ContODETthetaE} and Corollary~\ref{cor:17-05_13:35}.
\end{proof}
We end this section introducing the concept of linearized skew-product flow.
\begin{defn}\label{def:linearizedskew}
Let $f\in \LC$ be continuously differentiable with respect to $x$ for a.e. $t\in\R$ and with $L^1_{loc}$-equicontinuous $m$-bounds. Let $\Theta$ be defined as  in Remark~\ref{rmk:THETAjfunc} and denote by $J_xf\in \SC\big(\R^{N\times N}\big)$ the Jacobian of $f$ with respect to the coordinates~$x$.
\begin{itemize}
\item If $\Hu_\T=\mathrm{Hull}_{(\LC\times\TC,\T_\Theta\times\T_\Theta)}(f,  J_xf)$ and $\U$ is the subset of $\R\times\Hu_\T\times\R^N\!\!\times\R^N$ given by
\begin{equation*}
\U_\T=\bigcup_{\substack{(g,G)\in\Hu_\T\\ x_0\in\R^N}}\left\{ (t,g,G,x_0,y_0)\mid t\in I_{g,x_0},\,y_0\in\R^N\right \}\,,
\end{equation*}
then we call  \emph{a linearized skew-product flow} the map
\begin{equation*}\label{eq:skewp-lineT}
\begin{split}
\qquad\qquad\Psi_\T\colon \U_\T\subset \R\times\Hu_\T\times\R^N\!\!\times\R^N\quad&\to\quad\ \qquad\qquad\Hu_\T\times\R^N\!\!\times\R^N\\
\qquad\qquad(t,g,G, x_0, y_0)\qquad\quad&\mapsto\;\; \big(g_t,G_t, x(t,g,x_0), y(t,g,G, x_0, y_0)\big)\, ,
\end{split}
\end{equation*}
\item If $\Hu_\sigma=\mathrm{Hull}_{(\LC\times\WTC,\sigma_\Theta\times\sigma_\Theta)}(f,  J_xf)$, where $J_xf$ has $L^1_{loc}$-equicontinuous $m$-bounds, and if $\U_\sigma$ is the subset of $\R\times\Hu_\sigma\times\R^N\!\!\times\R^N$ given by
\begin{equation*}
\U_\sigma=\bigcup_{\substack{(g,G)\in\Hu\sigma\\ x_0\in\R^N}}\left\{ (t,g,G,x_0,y_0)\mid t\in I_{g,x_0},\,y_0\in\R^N\right \}\,,
\end{equation*}
then, we call  \emph{a $\sigma$-linearized skew-product flow} the map
\begin{equation*}\label{eq:skewp-lineS}
\begin{split}
\qquad\qquad\Psi_\sigma\colon \U_\sigma\subset \R\times\Hu_\sigma\times\R^N\!\!\times\R^N\quad&\to\quad\ \qquad\qquad\Hu_\sigma\times\R^N\!\!\times\R^N\\
\qquad\qquad(t,g,G, x_0, y_0)\qquad\quad&\mapsto\quad \big(g_t,G_t, x(t,g,x_0), y(t,g,G, x_0, y_0)\big)\, ,
\end{split}
\end{equation*}
\end{itemize}
\end{defn}
The use of the name \emph{linearized skew-product flow} is meaningful thanks to Theorem 6.1 in \cite{paper:LNO}. Moreover, one can easily check that  a slight generalization of the proof of such theorem gives meaning to the definition of \emph{$\sigma$-linearized skew-product flow}. Indeed, the weak topology $\sigma_\Theta$ used in Theorem \ref{thm:ContODETthetaE} is a good and weaker alternative. However, a stricter assumption on the $m$-bounds of the Jacobian of $f$ has to be assumed. In any case, one has that for every $(g,G)\in\Hu_\T$ ($(g,G)\in\Hu_\sigma$ resp.) and every $t\in I_{g,x_0}$
\begin{equation*}
\frac{\partial x(t,g,x_0)}{\partial x_0}\cdot  y_0= y(t,g,G,x_0, y_0)\,,
\end{equation*}
and therefore  in particular when $G\in\TC\big(\R^{N\times N}\big)\setminus\SC\big(\R^{N\times N}\big)$ ($G\in\WTC\big(\R^{N\times N}\big)\setminus\SC\big(\R^{N\times N}\big)$ resp.), i.e. when $g$ does not have continuous partial derivatives with respect to $x$ for almost every $t\in\R$.
\section{Exponential dichotomy and dichotomy spectrum}\label{expdich}
In this section we look more deeply into the properties of the linearized skew-product flows introduced at the end of last section. In particular we investigate the behavior of the solutions of the linear system when it has  exponential dichotomy and study its dichotomy spectrum. Firstly, let us state some assumptions and simplify  the notation.
\par\smallskip
Let $\Hu$ be either $\Hu_\T$ or $\Hu_\sigma$ as defined in Definition \ref{def:linearizedskew}, and assume that for each $(g,G)\in \Hu$.  the solutions of
\begin{equation*}
\begin{cases}
\dot x=g(t,x)\,, &x(0)=x_0\,,\\
\dot y=G(t,x)\,y\,, &y(0)= y_0\,,
\end{cases}
\end{equation*}
are globally defined or, equivalently, $x(t,g,x_0)$ is globally defined. As a consequence, the linearized skew-product flow is defined on the whole $\R\times\Hu\times\R^N\!\!\times\R^N$. \par\smallskip
Moreover, denoting by $\Omega=\Hu\times \R^N$, the continuous skew-product flow $\Psi$ in definition \ref{def:linearizedskew} can be read as a continuous linear skew-product flow
\begin{equation}\label{eq:skewp-linear}
\begin{array}{ccc}
\Psi\colon  \R\times\Omega\times\R^N &\to &\Omega\times\R^N \\[.1cm]
 \quad(t,\omega, y_0)&\mapsto& \big(\omega_t, y(t,\omega, y_0)\big)\, ,
\end{array}
\end{equation}
where the flow on the base  $\R\times\Omega\to \Omega$, $(t,\omega)\mapsto \omega_t$ is defined, for each $\omega=(g,G,x_0)$,  by $\omega_t=(g_t,G_t,x(t,g,x_0))$.
Additionally, consider the function $A:\Omega\to \R^{N\times N}$ defined as follows
\begin{equation*}
A(\omega)=\begin{cases}
\displaystyle \lim_{h\to 0}\frac{1}{h}\int_0^h G\big(s, x(s,g,x_0)\big)\, ds&\text{if the limit exists}\\[.1pt]
0&\text{otherwise.}
\end{cases}
\end{equation*}
Notice that, in fact,
\begin{equation}\label{eq:08.06-12:50}
A(\omega_t)=G\big(t,x(t,g,x_0)\big) \qquad\text{for a.e. }t\in\R.
\end{equation}
Indeed, fixed $\omega_t=\big(g_t,G_t,x(t,g,x_0)\big)$ one has
\begin{equation*}
\begin{split}
\lim_{h\to 0}\frac{1}{h}\int_0^h \!\!G_t\Big(s, x\big(s,g_t,x(t,g,x_0)\big)\Big)\, ds&=\lim_{h\to 0}\frac{1}{h}\int_0^h \!\!G\big(s+t, x(s+t,g,x_0)\big)\, ds\\
=\lim_{h\to 0}\frac{1}{h}\int_t^{t+h} \!\!G\big(u, x(u,g,x_0)\big)\, du&=G(t,x(t,g,x_0)) \qquad\text{for a.e. }t\in\R,
\end{split}
\end{equation*}
which implies \eqref{eq:08.06-12:50}.\par\smallskip
Then, the family of systems $\dot y=G(t,x(t,g,x_0))\,y$,  with $\omega=(g,G,x_0)\in \Omega$, can be written  as
\begin{equation}\label{eq:famsysA}
\dot y=A(\omega_t)\,y\,,\quad \omega\in\Omega\,,
\end{equation}
 and if $\Phi(t,\omega)$ denotes the fundamental matrix solution of the system corresponding to $\omega$ with $\Phi(0,\omega)={\rm I}_N$, we have that $y(t,\omega,y_0)=\Phi(t,\omega)\,y_0$.

\begin{defn} Let $I$ be one of the half-lines $(-\infty,0]$, $[0,\infty)$ or the real line $\R$ and let $\Delta$ be a subset of $\Omega$.
We say that the linear skew-product flow~\eqref{eq:skewp-linear}, or that the family \eqref{eq:famsysA}, \emph{has exponential dichotomy on $I$} over the set $\Delta$ if there are a family of continuous projections
$P\colon\Delta\to \mathcal L(\R^N,\R^N)$, $\omega\mapsto P(\omega)$, and constants $K\ge 1$ and $\alpha>0$, such that for every $s$, $t\in I$ and every $\omega\in\Delta$
\begin{equation}\label{eq:dichotomy}
\begin{split}
 &\left\|\Phi(t,\omega)\,P(\omega)\,\Phi^{-1}(s,\omega)\right\|\le K\,e^{-\alpha\,(t-s)}\qquad\qquad   \text{ if }\;\; t\ge s, \\
 & \left\|\Phi(t,\omega)\,\big({\rm I}_N-\,P(\omega)\big)\,\Phi^{-1}(s,\omega)\right\|
 \le K\,e^{\alpha\,(t-s)}\;  \quad \text{ if }\;\; t\le s.
\end{split}
\end{equation}
When $\Delta$ reduces to a point $\omega=(f,G,x_0)$, it is said that the corresponding system $\dot y=A(\omega_t)\,y$, i.e. $\dot y=G(t,x(t,g,x_0)\,y$ \emph{has exponential dichotomy on $I$}.
If $I=\R$ the interval will be omitted from the definition.
\end{defn}
\begin{defn} The set $\mathds{A}(\omega)$ will denote the alpha limit set of a point  $\omega=(g,G,x_0)\in\Omega$, that is,  $\widehat\omega=(\widehat g, \widehat G,\widehat x_0)\in\mathds{A}(\omega)$  if there is a sequence  $(t_n)_{n\in\N}$ in $\R$ such that $t_n\downarrow -\infty$ and $\widehat \omega=\lim_{n\to\infty} \omega_{t_n}$
in the corresponding product topology, i.e. $(\widehat g, \widehat G,\widehat x_0)=\lim_{n\to\infty} (g_{t_n},G_{t_n},x(t_n,g,x_0))$.\par\smallskip
Analogously $\widehat\omega=(\widehat g, \widehat G,\widehat x_0)$ belongs to the omega limit set $\mathds{O}(\omega)$   if there is a sequence  $(t_n)_{n\in\N}$ in $\R$ such that $t_n\uparrow \infty$ and $\widehat \omega=\lim_{n\to\infty} \omega_{t_n}$, i.e.
$(\widehat g, \widehat G,\widehat x_0)=\lim_{n\to\infty} (g_{t_n},G_{t_n},x(t_n,g,x_0))$.
Finally, $\mathds{H}(\omega)$ will denote the closure in $\Omega$ of the set $\{\omega_t=(g_t,G_t,x(t,g,x_0))\mid t\in\R\}$ for the corresponding product topology on $\Omega$.
\end{defn}
The next result shows how the exponential dichotomy of a particular system can be transferred to the exponential dichotomy  of the skew-product flow over its alpha limit set, its omega limit set or its hull.
\begin{prop}\label{prop:expo-over}
Let $\omega=(g,G,x_0)\in\Omega$.
\begin{itemize}
\item[(i)] If the linear system $\dot y=A(\omega_t)\,y$ has  exponential dichotomy on $(-\infty,0]$, then the skew-product flow~\eqref{eq:skewp-linear} has  exponential dichotomy over the alpha limit set $\mathds{A}(\omega)\subset \Omega$.
\item[(ii)] If the linear system $\dot y=A(\omega_t)\,y$ has exponential dichotomy on $[0,\infty)$, then the skew-product flow~\eqref{eq:skewp-linear} has  exponential dichotomy over the omega limit set $\mathds{O}(\omega)\subset \Omega$.
\item[(iii)] If the linear system $\dot y=A(\omega_t)\,y$ has  exponential dichotomy, then the skew-product flow~\eqref{eq:skewp-linear} has  exponential dichotomy over the whole hull $\mathds{H}(\omega)\subset \Omega$.
\end{itemize}
\end{prop}
\begin{proof} (i) Let  $P(\omega)$ be the projection corresponding to the exponential dichotomy on $(-\infty,0]$ for the system $\dot y=A(\omega_t)\,y$ and  define the family of
projections
\begin{equation}\label{eq:proj-r}
P(\omega_r)=\Phi(r,\omega)\,P(\omega)\,\Phi^{-1}(r,\omega)\,\quad \text{ for each } r\le 0\,.
\end{equation}
We deduce that $\,\left\|\Phi(t,\omega_r)\,P(\omega_r)\,\Phi^{-1}(s,\omega_r)\right\|=
\left\|\Phi(t+r,\omega)\,P(\omega)\,\Phi^{-1}(s+r,\omega) \right\|$ and $ \left\|\Phi(t,\omega_r)\,({\rm I}_N-\,P(\omega_r))\,\Phi^{-1}(s,\omega_r)\right\|=
 \left\|\Phi(t+r,\omega)\,({\rm I}_N-P(\omega))\,\Phi^{-1}(s+r,\omega) \right\|$
 and, consequently
\begin{equation}\label{eq:dicho-r}
\begin{split}
 &\left\|\Phi(t,\omega_r)\,P(\omega_r)\,\Phi^{-1}(s,\omega_r)\right\|\le K\,e^{-\alpha\,(t-s)}\qquad\qquad   \text{ if }\;\;s\le t\le -r, \\
 & \left\|\Phi(t,\omega_r)\,\big({\rm I}_N-\,P(\omega_r)\big)\,\Phi^{-1}(s,\omega_r)\right\|
 \le K\,e^{\alpha\,(t-s)}\;  \quad\text{ if }\;\; t\le s\le - r.
\end{split}
\end{equation}
Next we take $\widehat \omega\in \mathds{A}(\omega)$ with $\widehat \omega=\lim_{n\to\infty} \omega_{r_n}$ for a sequence $r_n\downarrow -\infty$. From~\eqref{eq:proj-r} and~\eqref{eq:dichotomy} we deduce that $\|P(\omega_r)\|\le K$ for every $r\le 0$ and hence, the sequence of projections $\{P(\omega_{r_n})\}_{n\in\N}$ admits a subsequence converging to a projection $P(\widehat\omega)$ whose uniqueness is guaranteed by Proposition~1.56 in~\cite{book:JONNF}. From this fact,~\eqref{eq:dicho-r} and the continuity of the flow on the base $\Omega$, we deduce that $\dot y=A(\widehat\omega_t)\,y$ admits exponential dichotomy with projection $P(\widehat\omega)$, that is
\begin{equation}\label{eq:dichoA}
\begin{split}
 &\left\|\Phi(t,\widehat\omega)\,P(\widehat\omega)\,\Phi^{-1}(s,\widehat\omega)\right\|\le K\,e^{-\alpha\,(t-s)}\qquad\qquad   \text{ if }\;\;t\ge s \\
 & \left\|\Phi(t,\widehat\omega)\,\big({\rm I}_N-\,P(\widehat\omega)\big)\,\Phi^{-1}(s,\widehat\omega)\right\|
 \le K\,e^{\alpha\,(t-s)}\;  \quad\text{ if }\;\; t\le s.
\end{split}
\end{equation}
In order to conclude the proof, we show the continuity of
\begin{equation*}
P:\mathds{A}(\omega)\to \mathcal L(\R^N,\R^N),\qquad\widehat \omega\mapsto P(\widehat \omega).
\end{equation*}
To the aim, consider a sequence $(\widehat\omega^n)_\nin$ in $\mathds{A}(\omega)$ converging to some $\widehat\omega\in\mathds{A}(\omega)$ and let us prove that $\big(P(\widehat\omega^n)\big)_\nin$ converges to $P(\widehat\omega)$. As before, from \eqref {eq:dichoA}, with $t=s$, one has that for all $\nin$: $\|P(\widehat\omega^n)\|\le K$ and thus, it converges, up to a subsequence, to a projection $\widehat P$, and again, from \eqref {eq:dichoA} for each $\nin$, the continuity of the flow on the base $\Omega$ and the uniqueness of the projection, we have that $\widehat P=P(\widehat\omega)$. Then, one has  the exponential dichotomy of the skew-product flow~\eqref{eq:skewp-linear} over~$\mathds{A}(\omega)$, as stated in (i). The proofs of (ii) and (iii) are omitted because analogous.
\end{proof}
We recall the definition of dichotomy spectrum, or Sacker-Sell spectrum, for one of the systems and for a subfamily of the family~\eqref{eq:famsysA}.
\begin{defn} Let $\omega\in\Omega$ be fixed. The \emph{dichotomy spectrum of  $\dot y=A(\omega_t)\,y$}, which will be denoted by $\Sigma(\omega)$,  is the set of
$\gamma\in\R$ such that $\,\dot y=\big(A(\omega_t)-\gamma\,{\rm I}_N\big)\,y\,$ does not have exponential dichotomy.  The resolvent set is $\rho(\omega)=\R\setminus \Sigma(\omega)$.
\end{defn}
\begin{defn}
Let $\Delta$ be a subset of $\Omega$. The \emph{dichotomy spectrum of the linear skew-product flow~\eqref{eq:skewp-linear} over $\Delta$}, denoted by $\Sigma(\Delta)$ is the set
of $\gamma\in\R$ such that the family $\,\dot y=\big(A(\omega_t)-\gamma\,{\rm I}_N \big)\,y\,$  does not have exponential dichotomy over $\Delta$.
\end{defn}
When $\Delta$ is an invariantly connected compact invariant set of $\Omega$, Sacker and Sell~\cite{sase} proved that $\Sigma(\Delta)$ is the union of $k$ compact intervals
\[\Sigma(\Delta)=[a_1,b_1]\cup\dots \cup [a_k,b_k]\,,\]
where $1\le k\le N$ and $a_1\le b_1<a_2\le b_2<\dots\le a_{k}\le b_k$.\par\smallskip
From Proposition~\ref{prop:expo-over} we deduce that $\Sigma(\omega)=\Sigma\big(\mathds{H}(\omega)\big)$ but $\mathds{H}(\omega)$ is not necessarily compact. Therefore, we follow Siegmund's approach in~\cite{paper:SS} to study the dichotomy spectrum $\Sigma(\omega)$.
He showed that either $\Sigma(\omega)$ is empty, or it is the whole $\R$, or there exists $k\in\N$, with $1\le k\le N$, such that
\[
\Sigma(\omega)= I_1 \cup[a_2,b_2]\cup\dots \cup [a_{k-1},b_{k-1}]\cup I_k\,,
\]
where $I_1$ is either $[a_1,b_1]$ or $(-\infty,b_1]$,  $I_k$ is either $[a_k,b_k]$ or $[a_k,\infty)$, and $a_1\le b_1<a_2\le b_2<\dots\le a_{k}\le b_k$. In addition, a decomposition of $\R\times\R^N$ in spectral manifolds holds, i.e.
\[\R\times\R^N=\mathcal{W}_0\oplus\dots\oplus\mathcal{W}_{k+1}\, ;\]
see~\cite{paper:SS} for details.
\par
He also proved that $\Sigma(\omega)=[a_1,b_1]\cup\dots \cup [a_k,b_k]$, with $1\le k\le N$, if and only if the system $\dot y=A(\omega_t)\,y$ has \emph{bounded growth}, i.e. there exist constants $K\ge 1$ and $\alpha\ge 0$ such that
\[ \|\Phi(t,\omega)\,\Phi^{-1}(s,\omega)\|\le K\,e^{\alpha\,|t-s|} \qquad \text{for}\quad t,s\in\R
\]
Moreover, in such a case the spectral manifolds $\mathcal{W}_0$ and $\mathcal{W}_{k+1}$ are trivial, i.e.  $\R\times\R^N=\mathcal{W}_1\oplus\dots\oplus\mathcal{W}_{k}$.\par\smallskip
We finish this section providing conditions under which Carath\'eodory systems have bounded growth and, as a consequence, the dichotomy spectrum $\Sigma(\omega)=\Sigma\big(\mathds{H}(\omega)\big)$ is a finite number of compact intervals as in the Sacker-Sell dichotomy spectrum.
\begin{prop}\label{prop:spectrum}
Let $\omega=(g,G,x_0)\in\Omega$ fixed.  Assume that  $G$ has $L^1_{loc}$-bounded $m$-bounds and that $x(\cdot,g,x_0)$  is  bounded. Then the system $\dot y=A(\omega_t)\,y$ has bounded growth and  $\Sigma(\omega)=[a_1,b_1]\cup\dots \cup [a_k,b_k]$.
\end{prop}
\begin{proof} Using the notation introduced in Theorem \ref{thm:05.07-13:44}, let $I_{g,x_0}$ be the interval of definition of $x(\cdot,g,x_0)$ and let $j\in\N$ such that $\|x(\cdot,g,x_0)\|_{L^\infty(I_{g,x_0})}\le j$. Since $G$, has  $L^1_{loc}$-bounded $m$-bounds, there is a positive constant $\alpha$ such that
\begin{equation}
\sup_{s\in\R}\int_0^1 m^j(r+s)\,dr\le \alpha
\label{eq:20-07_13:29}
\end{equation}
where $m^j$ is the $m$-bound of $G$ on $B_j$ satisfying the assumption of $L^1_{loc}$-boundedness. Consider $s,t\in I_{g,x_0}$ and, for simplicity, assume that $s\le t$, the other case being analogous. Notice that $\Phi(t,\omega)\,\Phi^{-1}(s,\omega)\,y_0=y(t,\omega,s,y_0)$. Then
\[|y(t,\omega,s,y_0)|\le |y_0|+\int_s^t \|G(u,x(u,g,x_0))\|\,|y(u,\omega,s,y_0)|\,du\,,\] and Gronwall inequality
provides
\begin{equation}\label{eq:ineqgrowth}
\begin{split}
|y(t,\omega,s,y_0)| & \le |y_0|\exp\bigg(\int_s^t\|G(u,x(u,g,x_0))\|\,du\bigg)\\
& = |y_0|\exp\bigg(\int_0^{t-s}\|G_s(r,x(r+s,g,x_0))\|\,dr\bigg)\,.
\end{split}
\end{equation}
Using the equivalence of the $2$-norm and the matrix norm, the inequality \eqref{eq:20-07_13:29}, and covering the interval $[0,t-s]$ with intervals of unit length, \eqref{eq:ineqgrowth} yields
\[|y(t,\omega,s,y_0)|\le K\, e^{\alpha\, (t-s)}|y_0|\]
for an appropriate constant $K\ge 1$, which finishes the proof.
\end{proof}
Recall that $\Omega$ is defined as either $ \Hu_\T\times\R^N$ or $ \Hu_\sigma\times\R^N$ where $\Hu_\sigma$ and $ \Hu_\T$ are defined in Definition \ref{def:linearizedskew}. Notice that in the case in which $\Omega= \Hu_\T\times\R^N$ both the assumptions of Proposition~\ref{prop:spectrum} are  necessary, whereas if $\Omega= \Hu_\sigma\times\R^N$, then the $L^1_{loc}$-boundedness for the $m$-bounds of $G$ is already implied by the $L^1_{loc}$-equicontinuity for the $m$-bounds of $J_x f$ thanks to Corollary \ref{cor:17-05_13:35}.
\section{Pullback and global attractors for Carath\'{e}odory ODEs}\label{sec-pullback}
This section deals with pullback and global attractors for Carath\'{e}odory ODEs as an application of the continuity of the skew-product flow. In particular we show how, starting from specific properties on the solutions of an initial problem $\dot x= f(t,x)$, it is possible to obtain the existence of a bounded pullback attractor for the processes induced by systems with vector field in either the alpha limit set of $f$, the omega limit set of $f$, or the whole hull of $f$. Furthermore, conditions for the existence of pullback and global attractors for the induced skew-product flow are also provided.
 \par\smallskip
 Let us recall two cases in which, depending on the properties on $f$ and on the used topology $\T$, one has a continuous skew-product flow on $\mathrm{Hull}_{(\LC,\T)}(f)\times\R^N$.\par\smallskip
\begin{itemize}
\item  \textbf{Case 1:} $f\in(\LC,\sigma_\Theta)$ with $L^1_{loc}$-equicontinuous $m$-bounds;  see Theorem~\ref{thm:contFlowTtheta}.\smallskip
\item \textbf{Case 2:} $f\in(\LC,\T_D)$ with $L^p_{loc}$-bounded $l$-bounds; see \cite [Theorem 5.9]{paper:LNO}.
\end{itemize}
In the rest of the section, $(\LC,\T)$ will denote any of the topological spaces outlined in the previous cases.
\subsection{Statement and definitions}
We set the environment in which we subsequently develop our results. Let $f$ be a function in $\LC$ and consider the nonautonomous initial value problems
\begin{equation}\label{eq:ivp}
\dot x=f(t,x) ,\quad x(r)=x_0,\qquad  \text{with } r\in \R \text{ and } x_0\in\R^N\!.
\end{equation}
If $f$ is such that for any $x_0\in\R^N$ and any $r\in\R$, the unique solution $x(\cdot,f,r,x_0)$ is defined on $[r,\infty)$, then a process is induced by
\begin{equation}\label{eq:process}
S_f(t+r,r)\,x_0=x(t+r,f,r,x_0)=x(t,f_r,x_0)\,,
\end{equation}
where $t\ge 0$, $r\in \R$ and $ x_0\in\R^N$. We recall that, given a metric space $(X,d)$ a process is a family of continuous maps $\{S(t,s)\mid t\ge s\}\subset \mathcal{C}(X)$ satisfying
\begin{itemize}
\item $S(t,t)\,x=x$ for every $t\in\R$ and $x\in X$.
\item $S(t,s)=S(t,r)\,S(r,s)$, for every $t\ge r \ge s$.
\item $(t,s,x)\mapsto S(t,s)\,x$ is continuous for every $t\ge s$ and $x\in X$.
\end{itemize}
\smallskip
The different types of ultimately bounded character of the solutions are defined in terms of the process as follows.
\begin{defn} Let $f\in\LC$. The solutions of $\dot x=f(t,x)$ are said to be
   \begin{itemize}
     \item[$\bullet$] \emph{uniformly ultimately bounded} if there is a positive constant $c>0$ such that for every $d>0$ there is a time $T(d)>0$ satisfying
     \[ \qquad|S_f(t+r,r)\,x_0|\le c \quad \text{ for every } r\in\R,\,t\ge T(d) \;\text{ and }\; |x_0|\le d\,;\]
      \item[$\bullet$] \emph{uniformly ultimately bounded on $[\tau,\infty)$}  if there is a positive constant $c(\tau)$ such that for every $d>0$ there is a  time $T(\tau,d)>0$ satisfying
       \begin{equation}\label{desi:ultimate}
         |S_f(t+r,r)\,x_0|\le c(\tau)\,,
       \end{equation}
whenever $r\ge \tau$, $t\ge T(\tau,d)$ and $|x_0|\le d$.
     \end{itemize}
\end{defn}
The definitions of pullback attractor for a process is also hereby recalled.
\begin{defn}\label{def:attractor}
A family of subsets $\mathcal A(\cdot)=\{\mathcal A(t) \mid t \in \R\}$ of the phase space $X$  is said to be a \emph{pullback attractor for the process $S(\cdot,\cdot)$} if
\begin{itemize}
\item[(i)] $\mathcal A(t)$ is compact for each $t\in\R$;
\item[(ii)] $\mathcal A(\cdot)$ is invariant, that is, $S(t, s)\,\mathcal A(s) = \mathcal A(t)$ for all $t\ge  s$;
\item[(iii)] for each $t\in \R$, $\mathcal A(t)$ pullback attracts bounded sets at time t, i.e. for any bounded set $B\subset X$ one has
    \[\lim_{s\to-\infty}\dist(S(t, s)\,B,\mathcal A(t)) = 0\,,\] where $\dist(A,B)$ is the \emph{Hausdorff semi-distance} between two nonempty sets $A$, $B\subset X$  i.e. $\dist(A,B) := \sup_{x\in A} \inf_{y\in B}d(x,y)$.
\item[(iv)] $\mathcal A$ is the minimal family of closed sets with property (iii).
\end{itemize}
\end{defn}
The notion of pullback absorbing family will also be necessary.
\begin{defn}\label{def:absorbing} Let  $S(\cdot,\cdot)$ be a process on a metric space $(X,d)$. A family of nonempty bounded sets $\{B(t)\subset X\mid t\in\R\}$ \emph{pullback absorbs bounded sets}, if for every $t\in\R$ and every bounded subset $D$ of $X$ there exists a time  $T(t,D)> 0$ such~that
\begin{equation*}
S(t,t-s)\,D \subset B(t) \quad \text{ for every }  s \ge T(t,D)\,.
\end{equation*}
We also say that $\{B(t)\subset X\mid t\in\R\}$ is a pullback bounded absorbing family. If for all $t\in\R$ one has $B(t)\equiv B\subset X$ we will say that $B$ is a pullback absorbing set.
\end{defn}
\begin{defn} A process $S(\cdot,\cdot)$ is \emph{pullback strongly bounded dissipative  on  $(-\infty,\tau]$} if there exists a family $\{B(t)\subset X\mid t\in\R\}$ of pullback bounded absorbing  sets such that for every bounded subset $D\subset X$, there is a time $T(\tau,D)>0$ so that
\begin{equation}\label{def:strpullbound}
S(t,t-s)\,D\subset B(\tau) \quad \text{ for every } t\le \tau  \text{ and } s\ge T(\tau,D)\,.
\end{equation}
\end{defn}
\begin{rmk}\label{rmk:bound-attrac}
In the finite dimensional case, the existence of a pullback bounded absorbing family ensures the existence of a pullback attractor  (see, e.g., Carvalho et al.~\cite{book:CLR} and Kloeden and Rasmussen~\cite{book:KR}). If in addition the family satisfies~\eqref{def:strpullbound}, then
the pullback attractor $\{\mathcal A(t) \mid t \in \R\}$ \emph{is bounded in the past} because
$\bigcup_{t\le \tau}\mathcal A(t)\subset B(\tau)\,.$\par\smallskip
Finally, if there is a bounded set $B$ such that for every bounded subset $D\subset X$ there is a time $T(D)>0$ so that
\begin{equation}\label{eq:bounabs}
S(t,t-s)\,D\subset B \quad \text{ for every } t\in\R \text{ and } s\ge T(D)\,,
\end{equation}
then there is a \emph{bounded pullback attractor}.
\end{rmk}
\begin{rmk}\label{rmk:unifultboun}
 Notice that condition~\eqref{eq:bounabs} is equivalent to the uniformly ultimately bounded character of the solutions of the system.
\end{rmk}
\par\smallskip
Finally, we recall the definitions of pullback and global attractor for a skew-product flow. Let $f\in\LC$ and let $\T$ be a topology in $\LC$ such that the induced local skew-product flow
\begin{equation}\label{skew-product}
\begin{split}
\Pi\colon  \U\subset \R\times\mathrm{Hull}_{(\LC,\T)}(f)\times\R^N\  &\to\ \ \mathrm{Hull}_{(\LC,\T)}(f)\times\R^N\\
\ (t,g,x_0)\quad \qquad \qquad &\mapsto\qquad \big(g_t, x(t,g,x_0)\big)\,,
\end{split}
\end{equation}
is continuous.
\begin{defn}
Assume that for any $g\in\mathrm{Hull}_{(\LC,\T)}(f)$ and any $x_0\in\R^N$, the solution $x(\cdot,g,x_0)$ of $\ \dot x=g(t,x),\  x(0)=x_0,\ $ is defined on $[0,\infty)$, i.e.  the skew-product semiflow \eqref{skew-product} is defined on $\R^+\!\!\times\mathrm{Hull}_{(\LC,\T)}(f)\times\R^N$.\vspace{0.2cm}
\begin{itemize}
\item A family $\widehat A=\{ A_g\mid g\in \mathrm{Hull}_{(\LC,\T)}(f)\}$ of nonempty, compact sets of $\R^N$ is said to be a \emph{pullback attractor for the skew-product semiflow} if it is invariant, i.e.
\[ x(t,g,A_g)=A_{g_t} \quad \text{ for each } t\ge 0 \;\text{ and }\; g\in \mathrm{Hull}_{(\LC,\T)}(f)\,,\]
and, for every nonempty bounded set $D$ of $\R^N$ and every $g\in \mathrm{Hull}_{(\LC,\T)}(f)$ one has
\begin{equation*}
\lim_{t\to\infty} \dist(x(t,g_{-t},D), A_g)=0\,,
\end{equation*}
where $\dist(A,B)$ denotes the \emph{Hausdorff semi-distance} of two nonempty sets $A$, $B$ of~$\R^N$.  A pullback attractor for the skew-product flow is said to be \emph{bounded} if
\begin{equation*}
\bigcup_{g\in\mathrm{Hull}_{(\LC,\T)}(f)} \!\!\!\!\!\!\!A_g\quad\text{is bounded.}
\end{equation*} \vspace{0.2cm}
\item A compact set $\mathcal A$ of $\mathrm{Hull}_{(\LC,\T)}(f)\times \R^N$ is said to be a \emph{global attractor for the skew-product semiflow} if it is the maximal nonempty compact subset of $\mathrm{Hull}_{(\LC,\T)}(f)\times \R^N$ which is $\Pi$-invariant, i.e.
\begin{equation*}
 \Pi(t,\mathcal A)=\mathcal A \quad \text{ for each }\; t\ge 0\,,
\end{equation*}
and attracts all compact subsets $\mathcal D$ of $\mathrm{Hull}_{(\LC,\T)}(f)\times \R^N$, i.e.
\[\lim_{t\to\infty} \dist(\Pi(t,\mathcal D),\mathcal A)=0\,,\]
where now $\dist(\mathcal B,\mathcal C)$ denotes the \emph{Hausdorff semi-distance} of two nonempty sets $\mathcal B$, $\mathcal C$ of $\mathrm{Hull}_{(\LC,\T)}(f)\times \R^N$.
\end{itemize}
\end{defn}
\begin{rmk}
Analogous definitions hold when we change $\mathrm{Hull}_{(\LC,\T)}(f)$ by the alpha limit set $\mathds A(f)$  or by the omega limit set $\mathds O(f)$ in~\eqref{skew-product} and consider the corresponding skew-product semiflows.
\end{rmk}
\subsection{General results for processes and skew-product semiflows}\label{sub:general}
The next result  gives conditions under which, given a pullback attractor bounded in the past for the process induced by  $\dot x =f(t,x)$, with $f\in\LC$, one has the existence of a bounded pullback attractor for the process induced by $\dot x= g(t,x)$, where $g\in\LC$ is any function in the alpha limit set  $\mathds{A}(f)$.
\begin{thm}\label{thm:alphalimit}
Let $f$ be in $\LC$ and $\T$ be a topology such that the induced local skew-product flow~\eqref{skew-product} is continuous, and assume that for any $g\in\{f_s\mid s\leq 0\}\cup\mathds{A}(f)$, and $x_0\in\R^N$, the solution $x(\cdot,g,x_0)$ of $\ \dot x=g(t,x),\  x(0)=x_0,\ $ is defined on $[0,\infty)$.
If there is a $\tau\in\R$ for which the process $S_f(\cdot,\cdot)$ is strongly pullback bounded dissipative  on $(-\infty,\tau]$, and if $g$ is any function in $\mathds{A}(f)$, then the solutions of $\dot x=g(t,x)$ are uniformly ultimately bounded.  In particular, the induced process $S_g(\cdot,\cdot)$ has a bounded pullback attractor.
\end{thm}
\begin{proof} Let  $D$ be a bounded set.  By hypothesis, there are $c=c(\tau)>0$  and $T(D)=T(\tau,D)>0$ such that for each $x_0\in D$ one has
\begin{equation*}
|S_f(t,t-s)\,x_0|=|x(s,f_{t-s},x_0)|\le c \quad \text{for } t\le \tau \text{ and } s\ge T(D)\,.
\end{equation*}
If $g=\lim_{n\to\infty} f_{t_n}$ with $t_n\downarrow -\infty$, then we have $g_{t-s}=\lim_{n\to\infty} f_{t_n+t-s}$ and by the continuity of the semiflow
\[|S_g(t,t-s)\,x_0)|=|x(s,g_{t-s},x_0)|=\left|\lim_{n\to\infty}x(s,f_{t_n+t-s},x_0)\right|\,.\]
Finally, there exists $n_0\in\N$ such that, if $n\ge n_0$, then $t_{n}+t\le\min\{ 0,\tau\}$, and thus \[ |S_g(t,t-s)\,x_0|\le c\quad \text{for every } t\in\R \text{ and } s\ge T(D)\,.\]
Therefore, from Remark~\ref{rmk:unifultboun} the solutions of $\dot x=g(t,x)$ are uniformly ultimately bounded and, as stated in Remark~\ref{rmk:bound-attrac}, a bounded pullback attractor exists.
\end{proof}
Analogously, we give conditions to have a bounded pullback attractor for the process induced by $\dot x= g(t,x)$, when $g\in\LC$ is any function in the omega limit set~$\mathds{O}(f)$.
\begin{thm}\label{th:boundedpullback}
Let $f$ be in $\LC$ and $\T$ be a topology such that the induced local skew-product flow \eqref{skew-product} is continuous, and assume that for any $g\in\{f_s\mid s\geq 0\}\cup\mathds{O}(f)$, and $x_0\in\R^N$, the solution $x(\cdot,g,x_0)$ of $ \dot x=g(t,x),\  x(0)=x_0,\ $ is defined on $[0,\infty)$.  If there is a $\tau\in\R$ for which the solutions of $\dot x=f(t,x)$ are uniformly ultimately bounded on $[\tau,\infty)$, and $g$ is any function in $\mathds{O}(f)$, then the solutions of $\dot x=g(t,x)$ are uniformly ultimately bounded and the induced process $S_g(\cdot,\cdot)$ has a bounded pullback attractor.
\end{thm}
\begin{proof}
From~\eqref{desi:ultimate}  it holds
\[ |S_f(t+s,s)\,x_0|=|x(t+s,f,s,x_0)|=|x(t,f_{s},x_0)|\le c(\tau)\]
 if $s\ge \tau$, $t\ge T(\tau,d)$  and  $|x_0|\le d$.
Since $g\in\mathds{O}(f)$ there is a sequence $t_n\uparrow\infty$ with $\lim_{n\to\infty}f_{t_n}=g$. Thus, $g_{r}=\lim_{n\to\infty}f_{t_n+r}$ and by the continuity of the solutions
\[|S_g(t+r,r)\,x_0|=|x(t,g_{r},x_0)|=\left| \lim_{n\to\infty} x(t,f_{t_n+r},x_0)\right|\,.
\]
Since there is $n_0\in\N$ such that, if $n\ge n_0$, then $t_n+r\ge\max\{0, \tau\}$,  we conclude~that
\[ |S_g(t+r,r)\,x_0|\le c(\tau) \quad \text{ whenever } r\in\R\,,t\ge T(\tau,d)\, \text{ and }\, |x_0|\le d\,,\]
that is, the solutions of $\dot x=g(t,x)$ are uniformly ultimately bounded, as claimed.
As in Theorem~\ref{thm:alphalimit}, from Remarks~\ref{rmk:unifultboun} and~\ref{rmk:bound-attrac} we obtain the thesis.
\end{proof}
Finally, we give conditions to have a bounded pullback attractor for the process induced by $\dot x= g(t,x)$, when $g\in\LC$ is any function in the hull of $f$.
\begin{thm}\label{thm:pullbackHull}
Let $f$ be in $\LC$ and $\T$ be a topology such that the induced local skew-product flow \eqref{skew-product} is defined on $\R^+\!\!\times\mathrm{Hull}_{(\LC,\T)}(f)\times\R^N$ and it is continuous.
 If there is a  pullback bounded absorbing set $B$ satisfying~\eqref{eq:bounabs}, then, for any $g\in\mathrm{Hull}_{(\LC,\T)}(f)$, one has that the solutions of $\dot x=g(t,x)$ are uniformly ultimately bounded  and the induced process $S_g(\cdot,\cdot)$ has a bounded pullback attractor.
\end{thm}
\begin{proof}
First, notice that $\mathrm{Hull}_{(\LC,\T)}(f)=\mathds{A}(f)\cup\mathds{O}(f)\cup \{f_\tau\mid \tau\in\R\}$. Moreover,  condition~\eqref{eq:bounabs} implies that the assumptions of Theorems~\ref{thm:alphalimit} and~\ref{th:boundedpullback} are satisfied, as shown in Remarks~\ref{rmk:unifultboun} and~\ref{rmk:bound-attrac}. Therefore, if $g\in\mathds{A}(f)$ (resp. $g\in\mathds{O}(f)$) the result follows from Theorem~\ref{thm:alphalimit} (resp. Theorem~\ref{th:boundedpullback}).  If $g$ is $f$, or one of its time-translations, the uniformly ultimately bounded character of the solutions cames  again from Remark~\ref{rmk:unifultboun}, which together with Remark~\ref{rmk:bound-attrac} allows to end the proof.
\end{proof}
The next result provides the existence of a pullback attractor as well as a global attractor (when $\mathrm{Hull}_{(\LC,\T)}(f)$ is compact) of the skew-product semiflow~\eqref{skew-product} and the relation between them.
We denote by $x(t,f,D)$ the subset of $\R^N$ given by $\{x(t,f,x_0)\mid x_0\in D\}$.
\begin{thm}\label{thm:skpHull} Let $f$ be in $\LC$ and $\T$ be a topology such that the induced skew-product semiflow \eqref{skew-product} is defined on $\R^+\!\!\times\mathrm{Hull}_{(\LC,\T)}(f)\times\R^N$ and it is continuous.
 Assume that there is a bounded set $B \subset\R^N$ such that for each nonempty bounded set $D$ there is a time $T(D)$ such that
\begin{equation}\label{conditionHull}
x(t,f_s,D) \subset B  \quad \text{ whenever }\; t\ge T(D)
\end{equation}
for every $s\in\R$.  Then
\begin{itemize}
\item[(i)]  there is a unique bounded pullback attractor $\widehat A=\{ A_g\mid g\in \mathrm{Hull}_{(\LC,\T)}(f)\}$ of the skew-product semiflow~\eqref{skew-product} given by
    \[A_g=\bigcap_{\tau\ge 0}\; \overline{\bigcup_{t\ge \tau}x(t,g_{-t}, B)}\quad \text{for each }\; g\in\mathrm{Hull}_{(\LC,\T)}(f)\,, \]
\item[(ii)] if $\mathrm{Hull}_{(\LC,\T)}(f)$ is compact, there is a global attractor of the skew-product semiflow~\eqref{skew-product} given by
    \[ \mathcal A=\bigcap_{\tau\ge 0}\; \overline{\bigcup_{t\ge \tau} \Pi(t, \mathrm{Hull}_{(\LC,\T)}(f)\times B)}\,=\bigcup_{g\in \mathrm{Hull}_{(\LC,\T)}(f)}\left\{\{ g\}\times A_g\right\}\,.\]
\end{itemize}
\end{thm}
\begin{proof}
First, from the continuity of the skew-product flow, we deduce that
\begin{equation*}
x(t,g,D) \subset B \quad \text{ for every }\; t\ge T(D) \;\text { and every }\; g\in \mathrm{Hull}_{(\LC,\T)}(f).
\end{equation*}
Therefore, among other references, (i) follows from Theorem~3.20 of~\cite{book:KR}. The existence of a global attractor $\mathcal A$ under the compactness of the base $\mathrm{Hull}_{(\LC,\T)}(f)$ follows from~Theorem 2.2 of Cheban~\emph{et al.}~\cite{paper:CKS} and, as shown in Theorem 16.2 of~\cite{book:CLR}, $A_g$ is the section of $\mathcal A$ over $g$, that is ${\mathcal A}=\bigcup_{g\in \mathrm{Hull}_{(\LC,\T)}(f)}\left\{\{ g\}\times A_g\right\}$, which finishes the proof.
\end{proof}
\begin{rmk} Notice that~\eqref{conditionHull}  is equivalent to ~\eqref{eq:bounabs}, that is, the process induced by $f$ has a  pullback bounded absorbing set $B$.
\end{rmk}
\begin{rmk} Under the assumptions of Theorem 5.13(i-ii), in general the pullback attractors $\{A_{g_t} \mid t \in \R\}$ with $g \in \mathrm{Hull}_{(\LC,\T)}(f)$ have no forward attraction properties for the corresponding processes. However the global attractor $\mathcal A$ always exhibits collective properties of forward attractivity (see Caraballo \emph{et al.}~\cite{paper:CLO}).
\end{rmk}
Finally, from  Theorems~\ref{thm:alphalimit},~\ref{th:boundedpullback} and~\ref{thm:skpHull}
we obtain the corresponding results for the induced skew-product flow on
$\mathds{A}(f)\times\R^N$ and $\mathds{O}(f) \times\R^N$.
\begin{cor}\label{skpAlpha}
Let $f\in\LC$ and $\T$ be a topology such that the induced local skew-product flow \eqref{skew-product} is continuous, and assume that for any $g\in\{f_s\mid s\leq 0\}\cup\mathds{A}(f)$, and any $x_0\in\R^N$,  the solution $x(\cdot,g,x_0)$ of $\ \dot x=g(t,x),\  x(0)=x_0\ $ is defined on $[0,\infty)$. If there is a $\tau\in\R$ for which $S_f(\cdot,\cdot)$  is strongly pullback bounded dissipative  on $(-\infty,\tau]$, then {\rm(i)} and {\rm (ii)} of {\rm Theorems~\ref{thm:skpHull}} hold for the skew-product flow on $\mathds{A}(f)\times \R^N$.
\end{cor}
\begin{cor}\label{skpOmega}
Let $f\in\LC$ and $\T$ be a topology such that the induced local skew-product flow \eqref{skew-product} is continuous, and assume that for any $g\in\{f_s\mid s\geq 0\}\cup\mathds{O}(f)$, and any $x_0\in\R^N$ the solution $x(\cdot,g,x_0)$ of  $ \dot x=g(t,x),\  x(0)=x_0\ $ is defined on $[0,\infty)$.  If there is a $\tau\in\R$ for which the solutions of $\dot x=f(t,x)$ are uniformly ultimately bounded on $[\tau,\infty)$, then {\rm(i)} and {\rm(ii)} of {\rm Theorems~\ref{thm:skpHull}} hold for the skew-product flow on $\mathds{O}(f)\times \R^N$.
\end{cor}
\section{Comparison methods for Carath\'{e}odory ODEs}\label{examples}
This section provides sufficient conditions under which the abstract results of subsection~\ref{sub:general} can be applied. In fact, several types of attractors, both for the induced process and the induced skew-product flow, are obtained. In the first subsection the size of the solutions of a Carath\'eodory differential system $\dot x=f(t,x)$ is compared with the size of the solutions of a scalar linear equation, while in the second subsection a comparison with a system of linear Carath\'eodory equations is carried out.
\subsection{Comparison with a scalar Carath\'eodory  linear equation}\label{scalar}
 Consider a Carath\'eodory differential system $\dot x=f(t,x)$ and the condition below for $f\in\LC$: \par\smallskip
\begin{itemize}
\item[\textbf H$_1$:] there exist $\alpha(\cdot)$, $\beta(\cdot)\in L^1_{loc}$, with  $\beta(\cdot)$ non-negative, such that
\begin{equation*}
2\,\langle f(t,x),x\rangle\le \alpha(t)\,|x|^2+\beta(t)\qquad \text{for a.e. } (t,x)\in\R^{N+1}\,,
\end{equation*}
where $\langle \cdot,\cdot\rangle$ represents the scalar product in $\R^N$ .
\end{itemize}
\par\smallskip
\noindent This assumption implies the following inequality for the solutions  of $\dot x=f(t,x)$.
\begin{prop}\label{prop:aet}
Assume that {\rm\textbf H$_1$} holds. If $x(t)$ is a solution of $\dot x=f(t,x)$ defined on an interval $I$, then it satisfies
\begin{equation}\label{desi:ae-dissp}
2\,\langle f(t,x(t)),x(t)\rangle\le \alpha(t)\,|x(t)|^2+\beta(t)\, \quad \text{\rm{ for a.e. }} t\in I.
\end{equation}
\end{prop}
\begin{proof} Let  $V\subset \R\times\R^N$ be such that $\meas_{\R^{1+N}}\left(\R^{1+N}\setminus V\right)=0$ and
\begin{equation*}
2\,\langle f(t,x),x \rangle \le \alpha(t)\,|x|^2+\beta(t)\quad \text{ for all }\, (t,x)\in V .
\end{equation*}
Consider the set $E=\left\{(t,\ep)\in I\times B_1\mid \big(t, x(t)+\ep\big)\in V\right\}$, where $B_1$ is the closed ball of $\R^N$ centered at the origin and with radius $1$, and for any $t\in I$ denote by $E_t$ the section in $t$  of $E$, i.e. $E_t=\{ \ep\in B_1\mid (t,\ep)\in E \}$. Moreover, given $t\in I$ one has that $x(t) + (B_1\setminus E_t)\subset B_{r} \setminus V_t$ for some $r$, and hence $\meas_{\R^N}(B_1\setminus E_t)=0$ for almost every $t\in I$.
Then, applying Fubini's theorem twice, one has
\begin{equation*}
\meas_{\R}(I)\cdot \meas_{\R^N}(B_1)=\meas_{\R^{1+N}}(E)=\int_{\R^N} \meas_{\R}(E_\ep) \, d\ep \,,
\end{equation*}
where $E_\ep$ denotes the section of $E$ given for any fixed $\ep\in B_1$. Therefore, one has $\meas_{\R}(E_\ep)=\meas_{\R}(I)$ for almost every $\ep\in B_1$. Now, let $(\ep_n)_{\nin}\subset B_1$ be such that
 \begin{equation*}
\ep_n\xrightarrow{\nti}0\quad \text{ and }\quad \meas_{\R}(E_{\ep_n})=\meas_{\R}(I)\quad \forall \ n\in \N\,.
\end{equation*}
As a consequence, taking $J=\underset{n\in\N}{\cap} E_{\ep_n}$ we deduce that
\[2\,\langle f(t,x(t)+\ep_n),x(t)+\ep_n\rangle\le \alpha(t)\,|x(t)+\ep_n|^2+\beta(t)\quad \forall t\in J\,,\]
and as $n\to\infty$ we obtain~\eqref{desi:ae-dissp} because $\meas_{\R}(I)=\meas_{\R}(J)$.
\end{proof}
\begin{rmk}\label{rmk:26-05_19:11}
If $f\in\LC$ satisfies \textbf H$_1$ then, considering the Cauchy problem $\dot x =f(t,x)$, $x(t_0)=x_0$, and denoted by $x(\cdot)$ its solution, from~\eqref{desi:ae-dissp} and $|x(r)|^2=\langle x(r),x(r)\rangle$ one has that
\[ \frac{d}{dr}|x(r)|^2=2\,\langle x(r),f(r,x(r))\rangle\le \alpha(r)\,|x(r)|^2+\beta(r)\,,\quad \text{ for a.e. } r\in\R\,.\]
Then, a standard comparison argument yields
\begin{equation}\label{eq:solution-escalar}
|x(t)|^2  \le \exp\left(\int_{t_0}^t\alpha(u)\,du\right)|x_0|^2+ \int_{t_0}^t\beta(r)\exp\left(\int_r^t \alpha(u)\,du\right)dr \,.
\end{equation}
As a consequence, the solutions of such a differential system are defined on $[t_0,\infty)$  and thus a process $S_f(\cdot,\cdot)$ can be induced as in \eqref{eq:process}.
\end{rmk}
In addition to \textbf H$_1$ we also consider the following conditions:\par\smallskip
\begin{itemize}
\item[\textbf H$_2$:] the equation $\,\dot y=\alpha(t)\,y\,$ has exponential dichotomy on $(-\infty,0]$ with projection $P=\rm{Id}$, that is, there are  constants $\,\alpha_1>0$ and $K\ge 1$ such that
\begin{equation}\label{eq:dicho}
\exp\left(\int_s^t \alpha(u)\,du\right)\le K\,e^{-\alpha_1\,(t-s)} \quad \text{ for } s\le t\le 0\,;
\end{equation}
\item[\textbf H$_3$:] the set of functions $\{\beta_t(\cdot)\}_{t\in\R}$ is $L^1_{loc}$-bounded,
\end{itemize}
\par\smallskip
\noindent
Assumptions {\rm\textbf H$_1$}, {\rm\textbf H$_2$} and {\rm\textbf H$_3$} allow to obtain that the process $S_f(\cdot,\cdot)$ is  strongly pullback bounded dissipative  on $(-\infty,\tau]$ for all $\tau\in\R$.
\begin{thm}\label{thm:pullbound}
Consider $f\in\LC$ and assume that {\rm\textbf H$_1$}, {\rm\textbf H$_2$} and {\rm\textbf H$_3$} hold. Then the induced process $S_f(\cdot,\cdot)$ is  strongly pullback bounded dissipative on $(-\infty,\tau]$ for all $\tau\in\R$. Consequently, there exists a pullback attractor which is bounded in the past.
\end{thm}
\begin{proof}
First we check  from \textbf H$_2$ that there is a nondecreasing function $K(t)\ge 1$  such that
  \begin{equation}\label{eq:dicho-t}
    \exp\left(\int_s^r \alpha(u)\,du\right)\le K(t)\,e^{-\alpha_1\,(r-s)} \quad \text{ for  } s\le r\le t\,.
  \end{equation}
Let $t\ge 0$ and denote $N(t)=\exp\left(\int_0^t |\alpha(u)|\,du\right)$.  From~\eqref{eq:dicho} we deduce that
\[
\exp\left(\int_s^r \alpha(u)\,du\right)\le K\,e^{\alpha_1\,s} N(t)=K\,e^{\alpha_1\,r}\,N(t)\,e^{-\alpha_1\,(r-s)}\,,
\]
and~\eqref{eq:dicho-t} holds for $K(t):=K\,e^{\alpha_1\,t}\,N(t)\,$ for $t\ge 0$, and $K(t):=K$ for $t\le 0$.\par\smallskip
Let $D$ be a bounded set. Thus, there is a positive constant $d>0$ such that $\sup_{x\in D}|x|\le d$. We take $x_0\in D$, $t\in\R$, $s\ge 0$  and denote by $x(\cdot):=x(\cdot,f,t-s,x_0)$, i.e. the solution of the Cauchy problem $\dot x= f(t,x)\,,\ x(t-s)=x_0$.
In particular, since  $S_f(t,t-s)\,x_0=x(t,f,t-s,x_0)=x(t)$
from~\eqref{eq:solution-escalar} and~\eqref{eq:dicho-t}  we deduce that
\begin{equation}\label{desi-para-coro}
|S_f(t,t-s)\,x_0|^2 \le |x_0|^2\,K(t)\,e^{-\alpha_1\,s}+I(t,s)
\;
\end{equation}
where
\begin{align*}
I(t,s)& =\int_{t-s}^t\beta(r)\exp\left(\int_r^t \alpha(u)\,du\right)dr \le  K(t)\int_{-\infty}^t e^{-\alpha_1\,(t-r)}\, \beta(r)\,dr \\
& = K(t) \int_{-t}^{\infty} e^{-\alpha_1\,(t+u)}\, \beta(-u)\,du\,.
\end{align*}
Now, from \textbf H$_3$ there is a  $c_1>0$ such that $\int_t^{t+1}\beta(u)\,du\le c_1$ for every $t\in\R$ and, hence,  if we decompose
$[-t,\infty)\subset \bigcup^{\infty}_{j=0}[-t+j,-t+j+1]\,,$ we obtain
\begin{equation}\label{eq:04-07_13:32}
I(t,s) \le K(t)\sum_{j=0}^{\infty}\int_{-t+j}^{-t+j+1} e^{-\alpha_1j}\,\beta(-u)\,du
 \le  c_1\,K(t) \sum_{j=0}^{\infty} e^{-\alpha_1j}\le \frac{c_1\,K(t)}{1-e^{-\alpha_1}}\,,
\end{equation}
because $\alpha_1>0$. Therefore,  denoting by $\rho^{2}(t):=1+ c_1\,K(t)/(1-e^{-\alpha_1})$
one has
\[|S_f(t,t-s)\,x_0|^2\le d^2\,K(t)\,e^{-\alpha_1\,s}+I(t,s)\le \rho^{2}(t)\,,\]
provided that $s\ge \ln (d^2 \,K(t))/\alpha_1 :=T(t,D)>0$. \par\smallskip
Hence, $\{B_{\rho(t)}\mid t\in\R\}$ is a family of bounded absorbing sets. In addition, since the function $K(t)$ is nondecreasing, then $\rho(t)$ and  $T(t,D)$ are also nondecreasing. Therefore,  we deduce that
\[S_f(t,t-s)\,D\subset B_{\rho(\tau)} \quad \text{for } t\le \tau \text{ and } s\ge T(\tau,D),\]
and the process is  strongly pullback bounded dissipative on $(-\infty,\tau]$ for all $\tau\in\R$, as claimed. The existence of a pullback attractor bounded in the past follows from Remark~\ref{rmk:bound-attrac}.
\end{proof}
Consequently, an application of Theorem~\ref{thm:alphalimit} provides for each $g$ in the alpha limit set $\mathds{A}(f)$ the existence of a bounded pullback attractor for the process~$S_g(\cdot,\cdot)$.
\begin{cor}\label{coro:alpha}
 Let $f$ be in $\LC$ and $\T$ be a topology such that the induced local skew-product flow on $\mathrm{Hull}_{(\LC,\T)}(f)\times\R^N$ is continuous. Under assumptions {\rm\textbf H$_1$}, {\rm\textbf H$_2$} and {\rm\textbf H$_3$}, for each $g\in\mathds{A}(f)$ the solutions of $\dot x=g(t,x)$ are uniformly ultimately bounded.  In particular, the induced process $S_g(\cdot,\cdot)$ has a bounded pullback attractor.
\end{cor}
\begin{proof}
In order to apply Theorem~\ref{thm:alphalimit}, we only need to prove that for any $g\in\mathds{A}(f)$ and any $x_0\in\R^N$ the solution of $\ \dot x=g(t,x),\ x(0)=x_0\ $ is defined on $[0,\infty)$. From~\eqref{desi-para-coro} and~\eqref{eq:04-07_13:32} and recalling that the function $K(\cdot)$ is non decreasing, one has that for any~$d\ge0$, and $x_0\in\R^N$ with  $|x_0|<d$
\begin{equation}
|S_f(t,t-s)\,x_0|^2=|x(s,f_{t-s},x_0)|\le c^2(d)\qquad \text{for all } t\le0,\; s\ge 0 \,,
\label{eq:11/08_17:09}
\end{equation}
where $c^2(d)=K(0)\left(d^2+c_1/(1-e^{-\alpha_1})\right)$. Let us fix $s\in [0,b_{g,x_0})$ and take $g=\lim_{n\to\infty} f_{t_n}$ with $t_n\downarrow -\infty$. Notice that, for any $\nin$, one may write $x(s,f_{t_n},x_0)$ as $x(s,f_{(t_n+s)-s},x_0)$. Thus, considered $n_0\in\N$ such that $t_n+s\le 0$ for any $n\ge n_0$, one has that $x(s,f_{t_n},x_0)=x(s,f_{(t_n+s)-s},x_0)$ satisfies \eqref{eq:11/08_17:09} for any $n\ge n_0$ and, by the continuity of the flow, the sequence $\big(x(s,f_{t_n},x_0)\big)_\nin$ converges to $x(s,g,x_0)$. Therefore, we conclude that
\begin{equation*}
 |x(s,g,x_0)|\le c(d)\quad \text{for all } s\in[0,b_{g,x_0})\; \text{ and }\,x_0\in\R^N\text{ with } |x_0|<d\,.
\end{equation*}
As a consequence, one has that the solution $x(\cdot, g,x_0)$ of $\ \dot x=g(s,x),\ x(0)=x_0\ $ can not explode in finite time, i.e. it has to be defined on $[0,\infty)$. Otherwise it is easy to prove that a contradiction arises.
One concludes the proof applying Theorem~\ref{thm:alphalimit}.
\end{proof}
In order to have that for  all $\tau\in\R$ the solutions of $\dot x=f(t,x)$ are uniformly ultimately bounded on  $[\tau,\infty)$, we change hypothesis (\textbf H$_2$)~by
\par\smallskip
\begin{itemize}
\item[\textbf H$_2^*$:] the linear equation $\,\dot y=\alpha(t)\,y\,$ has exponential dichotomy on $(0,\infty]$ with projection $P=\rm{Id}$, i.e. there is an $\,\alpha_1>0$ and a constant $K\ge 1$ such that
    \begin{equation*}\label{eq:dicho+}
    \exp\left(\int_s^t \alpha(u)\,du\right)\le K\,e^{-\alpha_1\,(t-s)} \quad \text{ for every } 0\le s\le t\,.
    \end{equation*}
 \end{itemize}
\begin{thm}\label{thm:uniultboun}
Under conditions {\rm\textbf H$_1$}, {\rm\textbf H$_2^*$} and {\rm\textbf H$_3$}, for each fixed $\tau\in\R$ the solutions of $\dot x=f(t,x)$ are uniformly ultimately bounded on $[\tau,\infty)$.
\end{thm}
\begin{proof} As in Theorem~\ref{thm:pullbound} we can prove the existence of a non-increasing function $K(\cdot)\ge 1$  such that
  \begin{equation*}\label{eq:dicho-t+}
    \exp\left(\int_s^r \alpha(u)\,du\right)\le K(t_0)\,e^{-\alpha_1\,(r-s)} \quad \text{ for every } t_0\le s\le r\,,
  \end{equation*}
and hence,
\[
|S_f(t+t_0,t_0)\,x_0|^2 \le |x_0|^2K(t_0)\,e^{-\alpha_1 t_0}+ I(t,t_0)
\]
where
 \begin{align*}
 I(t,t_0)&=\int_{t_0}^{t+t_0}\beta(r)\exp\left(\int_r^{t+t_0} \alpha(u)\,du\right)dr\le K(t_0)\int_{t_0}^{t+t_0}\beta(r)\,e^{-\alpha_1\,(t+t_0-r)}\,dr\\
 &= K(t_0)\int_{-t_0}^{t-t_0}\beta(t-u)\,e^{-\alpha_1\,(t_0+u)}\,du\le K(t_0)\int_{-t_0}^{\infty}\beta(t-u)\,e^{-\alpha_1\,(t_0+u)}\,du\,.
 \end{align*}
 Again, as in Theorem~\ref{thm:pullbound}, from~\textbf H$_3$ we deduce that
 $I(t,t_0)\le c_1\,K(t_0)/(1-e^{-\alpha_1})$ and denoting $\;c^2(t_0):=1+ c_1\,K(t_0)/(1-e^{-\alpha_1})\;$ and $\;T(t_0,d):=(\ln (d^2\,K(t_0))/\alpha_1\,$, it holds
\[ |S_f(t+t_0,t_0)\,x_0|\le c(t_0) \quad \text{ whenever } t\ge T(t_0,d) \;\text{ and }\; |x_0|\le d\,,\]
 and the non-increasing character of $c(\cdot)$ and $T(\cdot,d)$ proves~\eqref{desi:ultimate} and finishes the proof.
\end{proof}
As in Corollary~\ref{coro:alpha}, from the inequalities obtained in Theorem~\ref{thm:uniultboun} we can check that for any $g\in\mathds{O}(f)$ and any $x_0\in\R^n$, the solution $x(\cdot,g,x_0)$ of the Cauchy problem $\ \dot x=g(t,x),\ x(0)=x_0\ $ is defined on $[0,\infty)$. Hence, an application of Theorem~\ref{th:boundedpullback} provides, for each $g$ in the omega limit set $\mathds{O}(f)$, the existence of a bounded pullback attractor for the induced process $S_g(\cdot,\cdot)$.
\begin{cor} Let $f$ be in $\LC$ and $\T$ be a topology such that the  induced local skew-product flow on $\mathrm{Hull}_{(\LC,\T)}(f)\times\R^N$ is continuous  and assume that conditions {\rm\textbf H$_1$}, {\rm\textbf H$_2^*$} and {\rm\textbf H$_3$} hold. Then, for each $g\in\mathds{O}(f)$ the solutions of $\dot x=g(t,x)$ are uniformly ultimately bounded
and the induced process $S_g(\cdot,\cdot)$ has a bounded pullback attractor.
\end{cor}
Next, substituting  \textbf H$_2$ with the stronger assumption below, we obtain  a  pullback bounded absorbing set $B$ satisfying~\eqref{eq:bounabs} which is what we need in the assumptions of Theorem \ref{thm:pullbackHull}.
 \par\smallskip
\begin{itemize}
\item[\textbf H$_2^\bullet$:] the linear equation $\,\dot y=\alpha(t)\,y\,$ has exponential dichotomy on $\R$ with projection $P=\rm{Id}$, i.e. there is an $\,\alpha_1>0$ and a constant $K\ge 1$ such that
    \begin{equation}\label{eq:dichoR}
    \exp\left(\int_s^t \alpha(u)\,du\right)\le K\,e^{-\alpha_1\,(t-s)} \quad \text{ for  } s\le t\,;
    \end{equation}
 \end{itemize}
\begin{thm}\label{thm:boundedpullbackHull}
Consider $f\in\LC$ and assume that {\rm\textbf H$_1$}, {\rm\textbf H$_2^\bullet$} and {\rm\textbf H$_3$} hold. Then there is a  pullback bounded absorbing set $B$ satisfying~\eqref{eq:bounabs} and, hence, the induced process $S_f(\cdot,\cdot)$ has a bounded pullback attractor.
\end{thm}
\begin{proof}
Since the constant $K$ in inequality~\eqref{eq:dichoR} holds for every $s\le t$, reasoning as in Theorem~\ref{thm:pullbound} one has that
\begin{equation}\label{eq:desi}
S_f(t,t-s)\,D\subset B_{\rho} \quad \text{for every } t\in\R \text{ and  } s\ge T(D),
\end{equation}
where $\rho^2=1+ c_1\,K/(1-e^{-\alpha_1})$,  $T(D)=\ln (d^2 \,K)/\alpha_1$, and~\eqref{eq:bounabs} holds with $B_\rho$, as stated. In particular, $S_f(\cdot,\cdot)$ has a bounded pullback attractor.
\end{proof}
As a consequence, an application of Theorem~\ref{thm:pullbackHull} provides the existence of a bounded pullback attractor for the process $S_g(\cdot,\cdot)$ induced by $g\in\mathrm{Hull}_{(\LC,\T)}(f)$. From the inequalities derived from~\eqref{eq:desi}, it is easy to check that the induced skew-product semiflow \eqref{skew-product} is defined on $\R^+\!\!\times\mathrm{Hull}_{(\LC,\T)}(f)\times\R^N$, and it is skipped.
\begin{cor} Let $f$ be in $\LC$ and $\T$ be a topology such that the induced local skew-product flow on $\mathrm{Hull}_{(\LC,\T)}(f)\times\R^N$ is continuous and assume that {\rm\textbf H$_1$}, {\rm\textbf H$_2^\bullet$} and {\rm\textbf H$_3$} hold. Then if $g\in\mathrm{Hull}_{(\LC,\T)}(f)$ one has that the solutions of $\dot x=g(t,x)$ are uniformly ultimately bounded,  and the induced process $S_g(\cdot,\cdot)$ has a bounded pullback attractor.
\end{cor}
We summarize the results for the existence of a pullback and a global attractor for the induced skew-product semiflow in the following remark.
\begin{rmk}  Under assumptions {\rm\textbf H$_1$}, {\rm\textbf H$_2^\bullet$} and {\rm\textbf H$_3$}, (i) and (ii) of Theorem~\ref{thm:skpHull} hold.  The same happens for Corollary~\ref{skpAlpha} (resp.~\ref{skpOmega}) when {\rm\textbf H$_1$}, {\rm\textbf H$_2$} and {\rm\textbf H$_3$} (resp. {\rm\textbf H$_1$}, {\rm\textbf H$_2^*$} and {\rm\textbf H$_3$}) are assumed.
\end{rmk}
\subsection{Comparison with a system of Carath\'eodory linear equations}\label{system}
 In this subsection we use a system of linear equations in order to control the vector field of our Carath\'eodory differential equation. Let us set some notation. In the following, for every $i=1,\dots,N$ the $i$th component of $x\in\R^N$ will be denoted by $x_i$. Moreover, if we write $x\ge0$ we mean that for all $i=1,\dots,N$ one has $x_i\ge0$, whereas we will write $x\gg 0$ if for every $i=1,\ldots N$ one has $x_i>0$. The space $\big(\R^N\big)^+$ will denote the set of points $x\in\R^N$ such that $x\ge0$. Analogously, the $i$th component of a vector function $f\colon\R\times\R^N\to\R^N$ will be denoted by $f_i$. We consider the new assumptions for $f\in\LC$:\par\smallskip
\begin{itemize}
 \item[\textbf A$_1$:] if $x\ge 0$ with $x_i=0$, then $f_i(t,x)\ge 0$ for a.e. $t\in\R$;\smallskip
 \item[\textbf A$_2$:] for a.e $(t,x)\in\R\times\big(\R^N\big)^+$
    \begin{equation*}
     f(t,x)\le A(t)\,x +b(t)\,,
     \end{equation*}
    where the functions $A(\cdot)=[a_{ij}(\cdot)]\in L^1_{loc}\big(\R^{N\times N}\big)$, $a_{ij}(\cdot)\ge 0$ for every $i\neq j$, $b(\cdot)\in L^1_{loc}\big(\R^N\big)$, and $b(t)\ge 0$ for every $t\in\R$;\smallskip
 \item[\textbf A$_3$:] the linear equation $\,\dot y=A(t)\,y\,$ has  exponential dichotomy on $(-\infty,0]$ with projection $P=\rm{Id}$, i.e. there is an $\,\alpha_1>0$ and a constant $K\ge 1$ such that
    \begin{equation*}\label{eq:dicho-sys}
    \|\Phi(t)\,\Phi^{-1}(s)\|\le K\,e^{-\alpha_1\,(t-s)} \quad \text{ for  } s\le t\le 0\,,
    \end{equation*}
    where $\Phi(t)$ is the fundamental matrix solution with $\Phi(0)={\rm I}_N$;\smallskip
 \item[\textbf A$_4$:] the set of functions $\{b_t(\cdot)\}_{t\in\R}$ is $L^1_{loc}\big(\R^N\big)$-bounded.
 \end{itemize}
 \smallskip
 The inequality in {\rm\textbf A$_2$} actually holds for the positive solutions of $\dot x=f(t,x)$. The proof is similar to the one of
 Proposition~\ref{prop:aet} and thus omitted.
\begin{prop}
Let $f$ be a function in $\LC$ satisfying {\rm\textbf A$_2$}. If $x(t)$ is a solution of $\dot x=f(t,x)$ defined on an interval $I$, with $x(t)\ge0$ for all $t\in I$, then
\begin{equation}\label{eq:desi-sys}
f\big(t,x(t)\big)\le A(t)\,x(t) +b(t)\, \quad \text{\rm{ for a.e. }} t\in I\,.
\end{equation}
\end{prop}
Conditions \textbf A$_1$ and \textbf A$_2$ imply that the system $\dot x=f(t,x)$ induces a continuous time process on~$\big(\R^N\big)^+$, as shown in the following result.
\begin{prop}\label{prop:positive}
Let $f$ be a function in $\LC$ and $x(t,f,t_0,x_0)$  the solution of~\eqref{eq:ivp} with $x_0\ge 0$.
\begin{itemize}
\item[\rm(i)] If  $f$ satisfies {\rm\textbf A$_1$}, then $x(t,f,t_0,x_0)\ge 0$ for every $t\ge t_0$  on its maximal interval of existence.
\item[\rm(ii)] If $f$ satisfies {\rm\textbf A$_1$} and {\rm\textbf A$_2$}, then $x(t,f,t_0,x_0)$
is defined on $[t_0,\infty)$.
\end{itemize}
As a consequence, under assumptions {\rm\textbf A$_1$} and {\rm\textbf A$_2$}  a continuous time process is induced on $\big(\R^N\big)^+$ by
\begin{equation}\label{eq:process+}
S_f(t,s)\,x_0=x(t,f,s,x_0)=x(t-s,f_s,x_0)\ge 0\,,\;\; \forall\; t\ge s\,\text{ and } x_0\in\big(\R^N\big)^+\,.
\end{equation}
\end{prop}
\begin{proof} (i) From the continuity with respect to initial data, it is enough to check that $\widetilde x(t)=x(t,f,t_0,x_0)\gg 0$ for $t\ge t_0$ whenever $x_0\gg 0$. Assume, on the opposite, that there is a first time $t_1>t_0$ for which one of the components vanishes. By simplicity of notation let the first one to be such a component. Then,
$\widetilde x_1(t)>0$ for $t\in[t_0,t_1)$ and $\widetilde x_1(t_1)=0$.\par
Notice that  $\widetilde x_1(t)$ is the solution of the scalar Carath\'{e}odory Cauchy value problem  $\dot y=g(t,y)\,,\;$ $y(t_0)=(x_0)_1$, with $g$  defined by
\[g(t,y(t))=f_1(t,y(t),\widetilde x_2(t),\ldots,\widetilde x_N(t))\,,\]
$f_1$ being the first component of the function $f$.
From {\rm\textbf A$_1$} we deduce that $g(t,0)=f_1(t,0,\widetilde x_2(t),\ldots,\widetilde x_N(t)))\ge 0$ for almost every $t$ in the maximal interval of definition of $\widetilde x$. Therefore, denoting by $n(\cdot)\equiv 0$, one has
\[\dot n(t)\le g(t,n(t)) \quad \text{ for a.e. } t\in[t_0,t_1]\,,\]
and the comparison theorem for Carath\'{e}odory scalar differential equations (see Olech and Opial~\cite{paper:olop}) yields
$\,n(t)\le y(t,t_0,g,0)\,$ for every $t\in[t_0,t_1]$. Moreover, since $ y(t,g,t_0,0)< y(t,g,t_0,(x_0)_1)=\widetilde x_1(t)$  we deduce that $0<\widetilde x_1(t)$ for every $ t\in[t_0,t_1]$, contradicting that $\widetilde x_1(t_1)=0$ and finishing the proof of (i).
\par
(ii) For simplicity of notation, let  $x(t)=x(t,f,t_0,x_0)$. From~\eqref{eq:desi-sys}  we deduce that
$\dot x(t)\le A(t)\,x(t) +b(t)$  for a.e.  $t$. Thus, since $a_{ij}(\cdot)\ge 0$ for  $i\neq j$, the linear system $\dot y=A(t)\,y+b(t)\,$ is quasi-monotone and  the comparison argument for Carath\'{e}odory systems, which is a consequence of the scalar one, yields
$x(t)\le y(t)$ for every $t\ge t_0$ where $y(t)$ denotes the solution  of $\dot y=A(t)\,y+b(t)\,$ with initial data $y(t_0)=x_0$. This fact together with the inequality $0\le x(t)$ shown in (i), finishes the proof.
\end{proof}
The following result provides the existence of a pullback attractor bounded in the past in $\big(\R^N\big)^+$.
\begin{thm}\label{thm:pullbound-sys} Let $f$ be a function in $\LC$ satisfying {\rm\textbf A$_1$}, {\rm\textbf A$_2$}, {\rm\textbf A$_3$} and {\rm\textbf A$_4$}. Then, the induced process~\eqref{eq:process+}  is strongly pullback bounded dissipative  on $(-\infty,\tau]$ for all $\tau\in\R$ and, as a consequence, there exists a pullback attractor which is bounded in the past.
\end{thm}
\begin{proof} As is Theorem~\ref{thm:pullbound},  from {\rm\textbf A$_3$} we deduce the existence of exponential dichotomy on $(-\infty,t]$ for any fixed $t\ge 0$. More precisely, there is a nondecreasing function $K(\cdot)\ge 1$  such that
  \begin{equation}\label{eq:dicho-t-sys}
    \|\Phi(r)\,\Phi^{-1}(s)\|\le K(t)\,e^{-\alpha_1\,(r-s)} \quad \text{ for } s\le r\le t\,.
  \end{equation}
Let $D$ be a bounded set of $\big(\R^N\big)^+$. Thus, there is a positive constant $d>0$ such that $\sup_{x\in D}|x|\le d$. We take $x_0\in D$, $s\ge 0$  and consider $x(r):=x(r,f,t-s,x_0)$, i.e. the solution of the Cauchy problem $\dot x(r)=f(r,x(r))\,$, $x(t-s)=x_0\,$. As in  Proposition~\ref{prop:positive}, we deduce that $0\le x(r)\le y(r)$
 for every $r\in [t-s,t]$ where $y(r)$ denotes the solution  of $\dot y=A(r)\,y+b(r)\,$ with initial data $y(t-s)=x_0$, that~is,
\[ 0\le x(t)\le \Phi(t)\,\Phi^{-1}(t-s)\,x_0+\int_{t-s}^t \Phi(t)\,\Phi^{-1}(r)\,b(r)\,dr\,.\]
Therefore, inequality~\eqref{eq:dicho-t-sys} provides
\[
|S_f(t,t-s)\,x_0|\le |x_0|\,K(t)\,e^{-\alpha_1\,s}+ K(t)\int_{t-s}^t e^{-\alpha_1\,(t-r)}\,|b(r)|\,dr\,,
\]
and the rest of the proof follows step by step the one of Theorem~\ref{thm:pullbound} and thus it is omitted.
\end{proof}
\begin{rmk} The part in condition \textbf A$_2$ which implies that the system $\dot y=A(r)\,y+b(r)\,$  is quasi monotone, i.e. $a_{ij}(\cdot)\ge 0$ for  $i\neq j$, can be substituted  by the quasi monotone condition for $\dot x=f(r,x)$, that is,
\[ f_i(r,x)\le f_i(r,z)\quad \text{ whenever } x\le z \text { and } x_i=z_i\,.\]
In this case, maintaining the notation of Theorem~\ref{thm:pullbound-sys}, we would obtain \[f(r,y(r))\le \dot y(r) \quad \text{ for a.e. } r\,,\]
which implies $x(r)\le y(r)$ for every $r\in [t-s,t]$, and the rest of the proof remains the same.
\end{rmk}
From Theorems~\ref{thm:pullbound-sys} and~\ref{thm:alphalimit}  we obtain we following result, whose proof is omitted because it is analogous to the one of Corollary~\ref{coro:alpha}.
\begin{cor}
 Let $f$ be a function in $\LC$ and $\T$ be a topology such that the induced skew-product flow on $\mathrm{Hull}_{(\LC,\T)}(f)\times(\R^N)^+$ is continuous. If $f$ satisfies {\rm\textbf A$_1$}, {\rm\textbf A$_2$}, {\rm\textbf A$_3$} and {\rm\textbf A$_4$}, and $g\in\mathds{A}(f)$, then the solutions of $\dot x=g(t,x)$ are uniformly ultimately bounded, and the  induced process  $S_g(\cdot,\cdot)$ on $\big(\R^N\big)^+$  has a bounded pullback attractor.
\end{cor}
Analogously if we change  hypothesis \textbf A$_3$ by
\begin{itemize}
\item[\textbf A$_3^*$:] the linear equation $\,\dot y=A(t)\,y\,$ has exponential dichotomy on $[0,\infty)$ with projection $P=\rm{Id}$, i.e. there is an $\,\alpha_1>0$ and a constant $K\ge 1$ such that
    \begin{equation*}\label{eq:dicho-sysvec}
    \|\Phi(t)\,\Phi^{-1}(s)\|\le K\,e^{-\alpha_1\,(t-s)} \quad \text{ for  } 0\le s\le t\,,
    \end{equation*}
    where $\Phi(t)$ is the fundamental matrix solution with $\Phi(0)={\rm I}_N$,\smallskip
\end{itemize}
we obtain a result analogous to Theorem~\ref{thm:uniultboun}, whose proof is omitted.
\begin{thm}
 Let $f$ be a function in $\LC$ satisfying {\rm\textbf A$_1$}, {\rm\textbf A$_2$}, {\rm\textbf A$_3^*$} and {\rm\textbf A$_4$}, for each fixed $\tau\in\R$ the solutions are uniformly ultimately bounded on $[\tau,\infty)$.
\end{thm}
 In particular, this implies that Theorem~\ref{th:boundedpullback} holds in this case and we deduce the following result.
\begin{cor}  Let $f$ be a function in $\LC$ and $\T$ be a topology such that the induced skew-product flow on $\mathrm{Hull}_{(\LC,\T)}(f)\times(\R^N)^+$ is continuous. If $f$  satisfies conditions {\rm\textbf A$_1$}, {\rm\textbf A$_2$}, {\rm\textbf A$_3^*$} and {\rm\textbf A$_4$} and $g\in\mathds{O}(f)$, then the solutions of $\dot x=g(t,x)$ are uniformly ultimately bounded, and the  induced process $S_g(\cdot,\cdot)$ on $\big(\R^N\big)^+$  has a bounded pullback attractor.
\end{cor}
Finally, if we change  hypothesis \textbf A$_3$ by \smallskip
\begin{itemize}
\item[\textbf A$_3^\bullet$:] the linear equation $\,\dot y=A(t)\,y\,$ has exponential dichotomy on $\R$ with projection $P=\rm{Id}$, i.e. there is an $\,\alpha_1>0$ and a constant $K\ge 1$ such that
    \begin{equation*}\label{eq:dicho-sysvecR}
    \|\Phi(t)\,\Phi^{-1}(s)\|\le K\,e^{-\alpha_1\,(t-s)} \quad \text{ for }  s\le t\,,
    \end{equation*}
    where $\Phi(t)$ is the fundamental matrix solution with $\Phi(0)={\rm I}_N$,\smallskip
\end{itemize}
we obtain a result analogous to Theorem~\ref{thm:boundedpullbackHull}, whose proof is omitted, and the corresponding corollary, consequence of Theorem~\ref{thm:pullbackHull}.
\begin{thm}
Consider $f\in\LC$ satisfying {\rm\textbf A$_1$}, {\rm\textbf A$_2$}, {\rm\textbf A$_3^\bullet$} and {\rm\textbf A$_4$}. Then there is a  pullback bounded absorbing set $B$ satisfying~\eqref{eq:bounabs} and, hence, the induced process~\eqref{eq:process+} has a bounded pullback attractor.
\end{thm}
\begin{cor}
Let $f$ be a function in $\LC$ and $\T$ be a topology such that the induced local skew-product flow on $\mathrm{Hull}_{(\LC,\T)}(f)\times(\R^N)^+$ is continuous. If $f$ satisfies {\rm\textbf A$_1$}, {\rm\textbf A$_2$}, {\rm\textbf A$_3^\bullet$} and {\rm\textbf A$_4$}, and $g\in\mathrm{Hull}_{(\LC,\T)}(f)$, then the induced process $S_g(\cdot,\cdot)$ has a bounded pullback attractor.
\end{cor}
Again, we summarize the results for the existence of a pullback and a global attractor for the  induced skew-product semiflow on $\mathrm{Hull}_{(\LC,\T)}(f)\times(\R^N)^+$ (resp. $\mathds{A}(f)\times(\R^N)^+$ and $\mathds{O}(f)\times(\R^N)^+$) in the following remark.
\begin{rmk}  Under assumptions {\rm\textbf A$_1$}, {\rm\textbf A$_2$}, {\rm\textbf A$_3^\bullet$} and {\rm\textbf A$_4$}, (i) and (ii) of Theorem~\ref{thm:skpHull} hold.  The same happens for the conclusions of Corollary~\ref{skpAlpha} (resp.~\ref{skpOmega}) when {\rm\textbf A$_1$}, {\rm\textbf A$_2$}, {\rm\textbf A$_3$} and {\rm\textbf A$_4$} (resp. {\rm\textbf A$_1$}, {\rm\textbf A$_2$}, {\rm\textbf A$_3^*$} and {\rm\textbf A$_4$}) are assumed.
\end{rmk}

\end{document}